%% file: glmsvs.tex
\documentclass[12pt,a4paper]{article}
   \usepackage[a4paper,tmargin=20mm,bmargin=15mm]{geometry}

\usepackage{amsmath,defns,amssymb,eulervm,reducecode,doi}
\usepackage{longtable}
\allowdisplaybreaks
\IfFileExists{ajr.sty}{\usepackage{ajr}}{}

\usepackage[backend=bibtex8 
    ,style=authoryear 
    ,backref=true 
    ]{biblatex}
\bibliography{bibexport}
\makeatletter
\let\citet\textcite
\let\citep\parencite
\def\cite{\@ifnextchar[{\parencite}{\textcite}}
\makeatother

\usepackage{pgfplots}
\pgfplotsset{compat=newest}
\usetikzlibrary{shapes.geometric}
\usepgfplotslibrary{external} 
\tikzset{external/force remake}

\title{Slowly varying, macroscale models emerge from microscale dynamics over multiscale domains}

\author{A.~J. Roberts
\thanks{School of Mathematical Sciences, University of Adelaide, South Australia.  \protect\url{mailto:anthony.roberts@adelaide.edu.au}
\textsc{orcid:}0000-0001-8930-1552 }
\and J.~E. Bunder
\thanks{School of Mathematical Sciences, University of Adelaide, South Australia.  \protect\url{mailto:judith.bunder@adelaide.edu.au}
\textsc{orcid:}0000-0001-5355-2288}
}

\date{\today}

\atBegin{definition}{\bigskip}
\atBegin{theorem}{\bigskip}
\atBegin{proposition}{\bigskip}
\atBegin{lemma}{\bigskip}
\atBegin{assumption}{\bigskip}

\Cal A\Cal B\Cal E\Cal G\Cal K\Cal L\Cal M
\Cal P\Cal Q\Cal S\Cal U\Cal W\Cal V\Cal Z

\Frak L\Frak V\Frak k\Frak u

\Bb C\Bb E\Bb G\Bb N\Bb X\Bb Y\Bb R\Bb U

\newcommand{\ydim}{Y}
\newcommand{\pxt}[2][X]{\ifx0#2%
  \else\ifx1#2(x-#1)%
  \else\frac{(x-#1)^{#2}}{#2!}%
  \fi\fi}
\newcommand{\pxvt}[2][\Xv]{\ifx0#2%
  \else\ifx1#2(\xv-#1)%
  \else\frac{(\xv-#1)^{#2}}{#2!}%
  \fi\fi}
\newcommand{\pyt}[2][Y]{\ifx0#2%
  \else\ifx1#2(y-#1)%
  \else\frac{(y-#1)^{#2}}{#2!}%
  \fi\fi}

\newcommand{\xivt}[1]{\ifx0#1 \else\ifx1#1\xiv
  \else\frac{\xiv^{#1}}{#1!}\fi\fi}

\newcommand{\Tr}[1]{#1^\dagger}

\Vec a\Vec b\Vec c\Vec d\Vec i\Vec j\Vec k
\Vec m\Vec n\Vec p\Vec q\Vec r\Vec x\Vec y\Vec u\Vec v
\Vec D\Vec C\Vec F\Vec R\Vec U\Vec X
\Vec{ell}\Vec{xi}\Vec{zeta}\Vec{phi}
\newcommand{\ov}{{\vec 0}}

\let\rvv\rv \def\rv{r}

\newcommand{\tr}{\ensuremath{\tilde r}}
\newcommand{\tu}{\ensuremath{\tilde u}}

\newcommand{\tL}{\ensuremath{\tilde\cL}}
\newcommand{\tK}{\ensuremath{\tilde\cK}}
\newcommand{\tN}{\ensuremath{\mathcal N}}

\newcommand{\tV}{\ensuremath{\tilde\cV}}
\newcommand{\tU}{\ensuremath{\tilde U}}
\newcommand{\tW}{\ensuremath{\tilde\cW}}

\def\ou\big(#1,#2,#3\big){{e^{\if#31\else#3\fi t}\star}#1\,}

\newcommand{\pdf}{\textsc{pdf}}

\newcommand{\Span}{\operatorname{span}}

\newcommand{\bA}{\ensuremath{\mathbf A}}

\newcommand{\des}{\textsc{de}s}

\atBegin{endexample}{\qed}

\begin{document}

\maketitle

\begin{abstract}
Many physical systems are well described on domains which are relatively large in some directions but relatively thin in other directions.  
In this scenario we typically expect the system to have emergent structures that vary slowly over the large dimensions.
For practical mathematical modelling of such systems we require efficient and accurate methodologies for reducing the dimension of the original system and extracting the emergent dynamics.
Common mathematical approximations for determining the emergent dynamics often rely on self-consistency arguments or limits as the aspect ratio of the `large' and `thin' dimensions becomes unphysically infinite.  
Here we build on a new approach, previously establish for systems which are large in only one dimension, which analyses  the dynamics at each cross-section of the domain with a rigorous multivariate Taylor series.  
Then centre manifold theory supports the local modelling of the system's emergent dynamics with coupling to neighbouring cross-sections treated as a non-autonomous forcing.  
The union over all cross-sections then provides powerful support for the existence and emergence of a centre manifold model global in the large finite domain.  
Quantitative error estimates are determined from the interactions between the cross-section coupling and both fast and slow dynamics. 
Two examples provide practical details of our methodology.
The approach developed here may be used to quantify the accuracy of known approximations, to extend such approximations to mixed order modelling, and to open previously intractable modelling issues to new tools and insights.
\end{abstract}

\tableofcontents

\section{Introduction}

System of large spatial extent in some directions and relatively thin extent in other dimensions are important in engineering and physics. 
Examples include thin fluid films, flood and tsunami modelling \cite[e.g.]{Noakes06, Bedient88, LeVeque2011}, pattern formation in systems near onset \cite[e.g.]{Newell69, Cross93, Westra2003},  wave interactions \cite[e.g.]{Nayfeh71b, Griffiths05},  elastic shells \cite[e.g.]{Naghdi72, Mielke88c, Lall2003}, and microstructured materials \cite[e.g.]{Romanazzi2016}.
There are many formal approaches to mathematically describe,  by means of modulation or amplitude equations, the relatively long time and space evolution of these systems~\cite[e.g.]{Vandyke87}.
This article develops a general approach to illuminate and enhance such practical approximations.
\cite{Roberts2013a} originally developed this approach for systems which have only one large dimension, and any number
of significantly smaller dimensions.
Here we consider the general case where the system has (finite) number of large dimensions and any number of thin dimensions.

The approach is to examine the dynamics in the locale around any cross-section.
We find that a truncated Taylor series---a Taylor multinomial---for local spatial structures is only coupled to neighbouring locales via the highest order resolved derivative. 
This coupling as is treated as an `uncertain forcing' of the local dynamics, and with Assumption~\ref{ass:spec}, we apply non-autonomous centre manifold theory \cite[e.g.]{Potzsche2006, Haragus2011} to prove the existence and emergence of a slowly varying local model.
The theoretical support provided by centre manifold theory applies for all cross-sections and so establishes existence and emergence of a centre manifold model globally over the spatial domain (Proposition~\ref{thm:svpde}) to form an `infinite' dimensional centre manifold \cite[e.g.]{Gallay93, Aulbach96, Aulbach2000}.
\autoref{sec:pmid} develops the methodology for a general linear \pde\ system defined on a general domain consisting of both `thin' and `large' dimensions.
In addition to rigorous proofs, Section \ref{sec:pmid}  also establishes a practical construction procedure based upon a multinomial generating function.

The new approach derives a novel quantitative estimate of the leading error, equation~\eqref{eq:pderemain}, obtained from the final term of the exact Taylor multinomial~\eqref{eq:Nintterm}.
Thus our approach not only provides new theoretical support for established methods such as the method of multiple scales, it extends the methods and provides new error estimates of the slowly varying model.
Interestingly, the theory is still valid in boundary layers and shocks, it is just that then the error terms are so large that the analysis in terms of centre-stable space variations is inappropriate.
However, we restrict the current analysis to linear \pde\ systems and will consider nonlinear systems in future work.

Two examples illustrate the general method and theory of \autoref{sec:pmid}.
First, \autoref{sec:mdhx} introduces the new approach with the example of a random walker who walks over a large plane but randomly changes direction among three directions.
We construct the emergent mean dynamics of the random walker across the plane, as a Fokker--Planck \pde, with a quantifiable error which agrees with the general form of the error~\eqref{eq:pderemain}.
The computer algebra code of \autoref{sec:camhe} implements the practical construction algorithm for this  example and confirms that the modelling extends to arbitrary order to rigorously derive a generalised Kramers--Moyal expansion for the random walker \cite[e.g.]{Pawula67}.

Second, \autoref{sec:ehmd2d} discusses the more complex example of a two dimensional heterogeneous diffusion problem---one with a spatially varying diffusivity in a cellular pattern. 
Via an ensemble of cellular phase-shifted problems, the heterogeneous diffusion problem is embedded in a family of problems which are homogeneous in the large dimensions and heterogeneous only in the thin dimensions of the  cellular ensemble. 
Then the local dynamics within a cellular cross-section leads to the ensemble averaged homogenisation of the original \pde{}.
This approach should underlie future rigorous modelling of pattern formation problems in multiple space dimensions.

The new methodology developed herein is \(\epsilon\)-free.  
Although the analysis is based upon a fixed reference equilibrium, crucially the subspace and centre manifold theorems guarantee the existence and emergence of models in a finite domain about this reference equilibrium.
Sometimes such a finite domain of applicability is large.
The only epsilons in this article appear in comparisons with other methodologies.

\section{Macroscale dynamics of a random walker}
\label{sec:mdhx}

This section introduces the novel approach in perhaps the simplest example system of the effective drift of a vacillating random walker.  
\autoref{sec:pmid} develops the approach for \emph{general} linear \pde{}s on general large but thin domains.

\begin{figure} 
\centering
\includegraphics{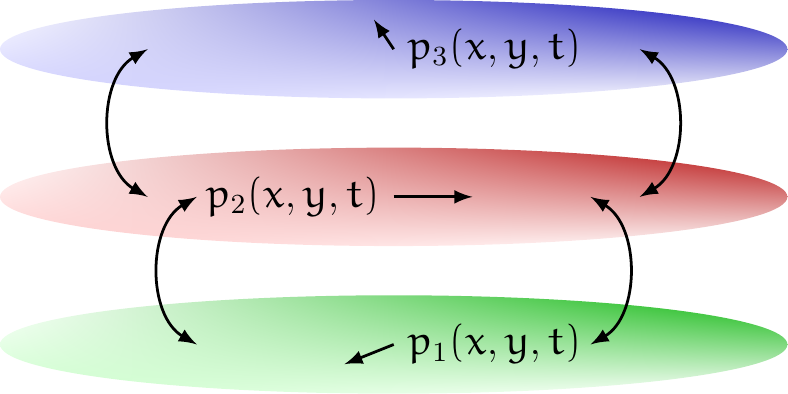}
\caption{Depending upon state of mind, a walker steps in one of three directions, indicated by different colours and indices $j=1,2,3$\,, and directions indicated by the intra-plane arrows. 
However, the walker randomly in time changes state of mind, indicated by the inter-plane double-headed arrows.}
\label{fig:heat2d}
\end{figure}

Consider the example of walker located somewhere in large domain in the \(xy\)-plane.
At any time the walker steps in one of three directions, as illustrated by different colours in \autoref{fig:heat2d}.
But the walker randomly changes directions, as also illustrated in \autoref{fig:heat2d}.
We model the probability that the walker is at position~\((x,y)\) and walking in direction~\(j\) at time~\(t\).

Let's explore the evolution of the probability density function (\pdf) for the random walker.
Let the \pdf\ be denoted by~$p_j(x,y,t)$ for direction states $j=1,2,3$\,.
To avoid some symmetry in the system, the walker cannot move directly between the first and third states, but must go indirectly via the second state.
Choose constant non-dimensional walking velocities in the three states of 
\(\vv_1=(1,1)\), \(\vv_2=(-1,0)\) and \(\vv_3=(1,-1)\).
The principle of conservation of probability gives that the three governing non-dimensional Fokker--Planck \pde{}s are
\begin{subequations}\label{eqs:hedim}%
\begin{eqnarray}
&&\D t{p_1}=-\D x{p_1}-\D y{p_1}+(p_2-p_1),
\label{eq:hedima}\\
&&\D t{p_2}={+\D x{p_1}}\phantom{{}+2\D y{p_2}}+(p_1-2p_2+p_3),
\label{eq:hedimb}\\
&&\D t{p_3}=-\D x{p_3}+\D y{p_3}+(p_2-p_3),
\label{eq:hedimc}
\end{eqnarray}
\end{subequations}
for \((x,y)\) in some large spatial domain~\XX.
For simplicity, in this section we assume the domain~\XX\ is convex---a restriction removed in the general analysis of \autoref{sec:pmid}.

An equivalent interpretation of the system of \pde{}s~\eqref{eqs:hedim} is that for a sort of heat exchanger when we view \(p_j(x,y,t)\)~as the temperature field in three adjacent plates, which is carried by some fluid flow at velocity~\(\vv_j\) in plate~\(j\) and diffuses between plates~\(j\), with only plate $j=2$ in contact with both plates $j=1,3$\,.
Our challenge then is to find a description of the large time, emergent, heat distribution. 
The equivalent challenge for the random walker is to describe his/her emergent probability distribution.

We focus upon the emergent solution of the \pde{}s~\eqref{eqs:hedim} in the interior of the domain~\XX{}.
This section ultimately finds that the mean \pdf, $u_0(x,y,t)=\tfrac13(p_1+p_2+p_3)$, satisfies the anisotropic Fokker--Planck\slash advection-diffusion \pde 
\begin{equation}
\D t{u_0}\approx-\frac13\D x{u_0}+\frac8{27}\DD x{u_0}+\frac23\DD y{u_0} 
\quad\text{for }(x,y)\in\XX\,.
\label{eq:appdiff}
\end{equation}
Many extant mathematical methods, such as homogenisation and multiple scales \cite[e.g.]{Engquist08, Pavliotis07}, would derive such an  \pde.
The main results of this section are to {rigorously} derive this \pde\ from a `local' analysis, 
complete with a novel quantitative error formula,
and with an innovative proof that the \pde\ arises as a naturally \emph{emergent} model from a wide variety of initial and boundary conditions.

The \pde{}s~\eqref{eqs:hedim} would have some boundary conditions specified on the boundary~\(\partial\XX\).
A future challenge is to determine the corresponding boundary conditions on~\(\partial\XX\) to be used with the model \pde~\eqref{eq:appdiff} in order to correctly predict the mean field~$u_0$ \cite[e.g.]{Roberts92c, Chen2016}.

The analysis is clearer in cross-state spectral modes.
Thus transform to fields
\begin{eqnarray}&&
u_0(x,y,t):=\tfrac13(p_1+p_2+p_3)\quad (\text{the mean field}),
\nonumber\\&&
u_1(x,y,t):=\tfrac12(p_1-p_3),\quad 
u_2(x,y,t):=\tfrac16(p_1-2p_2+p_3). \label{eq:udefs}
\end{eqnarray}
Equivalently, \(p_1=u_0+u_1+u_2\)\,, \(p_2=u_0-2u_2\) and \(p_3=u_0-u_1+u_2\)\,.
Then the governing \pde{}s~\eqref{eqs:hedim} become the separated `slow-fast' system
\begin{subequations}\label{eqs:hedimu}%
\begin{eqnarray}&&
\D t{u_0}=\phantom{-3u_2}-\frac13\D x{u_0}-\frac23\D y{u_1}-\frac43\D x{u_2}\,,
\label{eq:hedimua}
\\&&
\D t{u_1}=\,-\ u_1\ -\ \D y{u_0}\ -\ \D x{u_1}\ -\ \D y{u_2}\,,
\label{eq:hedimub}
\\&&
\D t{u_2}=-3u_2-\frac23\D x{u_0}-\frac13\D y{u_1}+\frac13\D x{u_2}\,.
\label{eq:hedimuc}
\end{eqnarray}
\end{subequations}
This form highlights that the difference fields~\(u_1\) and~\(u_2\) tend to decay exponentially quickly, but that interaction between gradients of the mean and difference fields generates other effects, effects that are crucial in deriving the macroscale Fokker--Planck\slash advection-diffusion \pde~\eqref{eq:appdiff}.

\subsection{In the interior}

Our approach expands the fields in their local spatial structure based around any station $(x,y)=(X,Y)$, and then the results apply to all stations.
As commented earlier, this approach is \(\epsilon\)-free.

\begin{figure}
\centering
\includegraphics{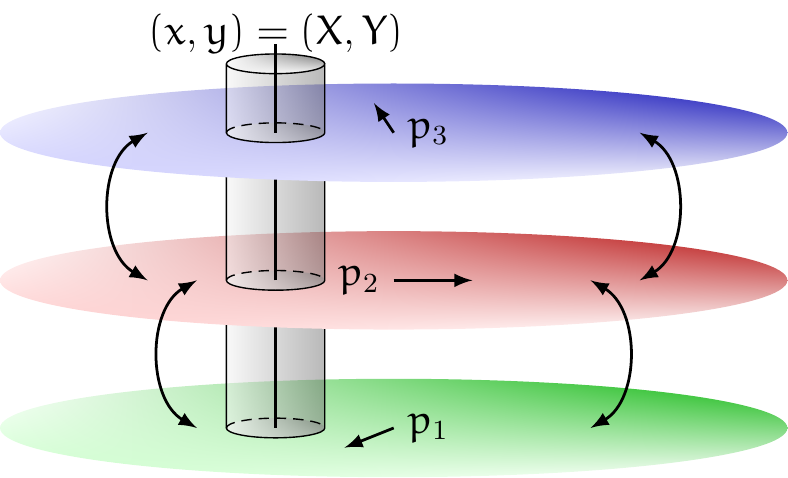}
\caption{schematic diagram of the random walker indicating that we focus on modelling the \pdf\ dynamics in the locale of a fixed station \((x,y)=(X,Y)\).
The vertical cylinder illustrates an example of the locale of the station \((x,y)=(X,Y)\).}
\label{fig:hestation}
\end{figure}  

Fix upon a station in~\XX\ at $(x,y)=(X,Y)$ as shown in \autoref{fig:hestation}.  
Consider the \pdf\ fields in the vicinity of $(x,y)=(X,Y)$, and denote the collective as vector \(u(x,y,t):=(u_0,u_1,u_2)\) (an unsubscripted~\(u\) denotes such a vector in~\(\RR^3\)).  
For all \((x,y)\in\XX\) (convex) invoke a multivariate Taylor's Remainder Theorem to express the fields exactly:
\begin{subequations}\label{eqs:utaylor1}%
\begin{eqnarray}
u&=&u^{00}(X,Y,t)+u^{10}(X,Y,t)\pxt1 
+u^{01}(X,Y,t)\pyt1
\nonumber\\&&{}
+u^{20}(X,Y,x,y,t)\pxt2
+u^{11}(X,Y,x,y,t)\pxt1\pyt1
\nonumber\\&&{}
 +u^{02}(X,Y,x,y,t)\pyt2\,,
\label{eqs:trt}
\end{eqnarray}
where,  assuming $u$~is twice differentiable in~\XX, the coefficient `derivatives' are
\begin{itemize}
\item  when \(m+n<2\)\,, 
\begin{equation}
u^{mn}(X,Y,t):=\left.\frac{\partial^{m+n}u}{\partial x^m\partial y^n}\right|_{(x,y)=(X,Y)}\,,
\end{equation}

\item whereas when \(m+n=2\)\,, 
\begin{equation}
u^{mn}(X,Y,x,y,t):=2\int_0^1 (1-s)
\left.\frac{\partial^{m+n}u}{\partial x^m\partial y^n}\right|_{(X,Y)+s(x-X,y-Y)}\,ds\,.
\label{eq:tayinrem}
\end{equation}
\end{itemize}
\end{subequations}

For definiteness and to avoid a combinatorial explosion of equations, this section truncates the Taylor expansion of~$u$ to the lowest order of interest, namely the quadratic approximation \(N=2\); \autoref{sec:camhe} lists computer algebra code that not only derives the results summarised here, but also derives corresponding results for arbitrarily specified truncation order~\(N\).

\paragraph{Local ODEs}

Substituting the Taylor expansion~\eqref{eqs:trt} into the governing \pde{}s~\eqref{eqs:hedimu} leads to 
a set of equations which are exact everywhere.
But in some places (namely near the station~\((X,Y)\)) the equations are useful in that remainder terms are negligibly small.
We derive a set of linearly independent equations for the coefficient functions~\(u_j^{mn}\) in~\eqref{eqs:utaylor1} simply by differentiation and evaluation at $(x,y)=(X,Y)$ (\autoref{sec:camhe}):
this process is almost the same as equating coefficients of terms in~\((x-X)^m(y-Y)^n\), but with care to maintain exactness one finds exact remainder terms.
To express the remainder terms, and because of the subsequent evaluation at the station \((x,y)=(X,Y)\), the symbols~\(u_j^{mn}\) such that $m+n=2$ denote \(u_j^{mn}(X,Y,X,Y,t)\);
further, the \(x\) and~\(y\) subscripts in the symbols~\(u_{jx}^{mn}\) and~\(u_{jy}^{mn}\) for \(m+n=2\) denote the definite but `uncertain third-order' derivatives 
\footnote{The `uncertain' derivatives~\(u_{jx}^{mn}\) and~\(u_{jy}^{mn}\) might appear to be simple third-order derivatives, but they are more subtle because of the integral~\eqref{eq:tayinrem}.}
\begin{equation}
u_{jx}^{mn}:=\left.\D x{u_{j}^{mn}}\right|_{(x,y)=(X,Y)}
\quad\text{and}\quad
u_{jy}^{mn}:=\left.\D y{u_{j}^{mn}}\right|_{(x,y)=(X,Y)}.
\label{eq:uxyd}
\end{equation}
Now, after substituting~\eqref{eqs:trt}, the various derivatives of~\eqref{eq:hedimua} evaluated at \((x,y)=(X,Y)\) lead to six \ode{}s for the six $u_0$~coefficients:
\begin{subequations}\label{eqs:5pde}%
\begin{eqnarray}&&
\dot u_0^{00}=-\tfrac13u_0^{10}-\tfrac23u_1^{01}-\tfrac43u_2^{10},\\&&
\dot u_0^{10}=-\tfrac13u_0^{20}-\tfrac23u_1^{11}-\tfrac43u_2^{20},\\&&
\dot u_0^{01}=-\tfrac13u_0^{11}-\tfrac23u_1^{02}-\tfrac43u_2^{11},\\&&
\dot u_0^{20}=-u_{0x}^{20}-\tfrac43u_{1x}^{11}
-4u_{2x}^{20}-\tfrac23u_{1y}^{20},\\&&
\dot u_0^{11}=
-\tfrac43u_{1y}^{11}-\tfrac13u_{0y}^{20}
-\tfrac43u_{2y}^{20}-\tfrac23u_{0x}^{11}
-\tfrac83u_{2x}^{11}-\tfrac23u_{1x}^{02},\\&&
\dot u_0^{02}=
-2u_{1y}^{02}-\tfrac23u_{0y}^{11}
-\tfrac13u_{0x}^{02}-\tfrac83u_{2y}^{11}
-\tfrac43u_{2x}^{02}.
\label{eq:5cpde}
\end{eqnarray}
Similarly, after substituting~\eqref{eqs:trt}, the various derivatives of~\eqref{eq:hedimub} evaluated at \((x,y)=(X,Y)\) lead to six \ode{}s for the six $u_1$~coefficients,
\begin{eqnarray}&&
\dot u_1^{00} =  - u_{0}^{01}  - u_{1}^{00}  - u_{1}^{10}  - u_{2}^{01} ,
\\&&
\dot u_1^{10} =  - u_{0}^{11}  - u_{1}^{10}  - u_{1}^{20}  - u_{2}^{11} ,
\\&&
\dot u_1^{01} =  - u_{0}^{02}  - u_{1}^{01}  - u_{1}^{11}  - u_{2}^{02} ,
\\&&
\dot u_1^{20} =  - u_{1}^{20}  - u_{0y}^{20}  - u_{2y}^{20}  - 2u_{0x}^{11}  - 3u_{1x}^{20}  - 2u_{2x}^{11} ;
\\&&
\dot u_1^{11} =  - u_{1}^{11}  - 2u_{0y}^{11}  - u_{1y}^{20}  - 2u_{2y}^{11}  - u_{0x}^{02}  - 2u_{1x}^{11}  - u_{2x}^{02} ,
\qquad
\\&&
\dot u_1^{02} =  - u_{1}^{02}  - 3u_{0y}^{02}  - 2u_{1y}^{11}  - 3u_{2y}^{02}  - u_{1x}^{02} ,
\label{eq:5dpde}
\end{eqnarray}%
and the various derivatives of~\eqref{eq:hedimuc} evaluated at \((x,y)=(X,Y)\) lead to six \ode{}s for the six $u_2$~coefficients,
\begin{eqnarray}&&
\dot u_2^{00} =  - \tfrac23u_{0}^{10}  - \tfrac13u_{1}^{01}  - 3u_{2}^{00}  + \tfrac13u_{2}^{10} ,
\\&&
\dot u_2^{10} =  - \tfrac23u_{0}^{20}  - \tfrac13u_{1}^{11}  - 3u_{2}^{10}  + \tfrac13u_{2}^{20} ,
\\&&
\dot u_2^{01} =  - \tfrac23u_{0}^{11}  - \tfrac13u_{1}^{02}  - 3u_{2}^{01}  + \tfrac13u_{2}^{11} ,
\\&&
\dot u_2^{20} =  - 3u_{2}^{20}  - \tfrac13u_{1y}^{20}  - 2u_{0x}^{20}  - \tfrac23u_{1x}^{11}  + u_{2x}^{20} .
\\&&
\dot u_2^{11} =  - 3u_{2}^{11}  - \tfrac23u_{0y}^{20}  - \tfrac23u_{1y}^{11}  + \tfrac13u_{2y}^{20}  
- \tfrac43u_{0x}^{11}  - \tfrac13u_{1x}^{02}  + \tfrac23u_{2x}^{11} ,
\qquad\\&&
\dot u_2^{02} =  - 3u_{2}^{02}  - \tfrac43u_{0y}^{11}  - u_{1y}^{02}  + \tfrac23u_{2y}^{11}  - \tfrac23u_{0x}^{02}  + \tfrac13u_{2x}^{02} ,
\label{eq:5epde}
\end{eqnarray}%
\end{subequations}

The functions \(u_{jx}^{mn}\) and~\(u_{jy}^{mn}\) for \(m+n=2\) that appear in~\eqref{eqs:5pde} are part of the closure problem for the local dynamics: 
the derivatives \(u_{jx}^{mn}\) and~\(u_{jy}^{mn}\) couple the dynamics at a station~$(X,Y)$ with the dynamics at neighbouring stations.
It is by treating terms in \(u_{jx}^{mn}\) and~\(u_{jy}^{mn}\) as `uncertain', time-dependent, inputs into the local dynamics that we notionally make the vast simplification in \emph{apparently} reducing the problem from one of an infinite dimensional dynamical system to a tractable eighteen dimensional system.

\subsection{The slow subspace emerges}
\label{sec:sse}

For a dynamical system approach to modelling the local dynamics, define the state vector $\uv(X,Y,t)=(u^{00}, u^{10}, u^{01}, u^{20}, u^{11}, u^{02})$ and group the eighteen \ode{}s~\eqref{eqs:5pde} into the matrix-vector system, of the form $\de t{\uv}=\cL\uv+\rvv(t)$,
\begin{subequations}\label{eqs:odes}%
\begin{equation}
\de t{\uv}=\underbrace{\begin{bmatrix} 
\fL_{0,0}&\fL_{1,0}&\fL_{0,1}\\
&\fL_{0,0}&&\fL_{1,0}&\fL_{0,1}\\
&&\fL_{0,0}&&\fL_{1,0}&\fL_{0,1}\\
&&&\fL_{0,0}\\
&&&&\fL_{0,0}\\
&&&&&\fL_{0,0}
\end{bmatrix}}_{\cL}\uv
+\underbrace{\begin{bmatrix} 0\\0\\0\\\rv_{2,0}\\
\rv_{1,1}\\\rv_{0,2} \end{bmatrix}}_{\rvv(t)}
\label{eq:odes}
\end{equation}
where blanks are zero-blocks, and where
\begin{equation}
\fL_{0,0}:=\begin{bmatrix} 0&0&0\\0&-1&0\\0&0&-3 \end{bmatrix},
\quad
\fL_{1,0}:=\begin{bmatrix} -\frac13&0&-\frac43\\0&-1&0\\-\frac23&0&+\frac13 \end{bmatrix},
\quad
\fL_{0,1}:=\begin{bmatrix} 0&-\frac23&0\\-1&0&-1\\0&-\frac13&0 \end{bmatrix},\label{eq:Lmat}
\end{equation}
and where \(\rv_{m,n}(t)\) in~\eqref{eq:odes} contain the appropriate terms from~\eqref{eqs:5pde} involving the definite but `uncertain'~\(u_{jx}^{mn}\) and~\(u_{jy}^{mn}\). 
\end{subequations}

The hierarchical structure of the matrix in~\eqref{eqs:odes} directly reflects the approach: the blocks~\(\fL_{m,n}\) directly encode the various derivatives appearing in the original governing \pde{}s~\eqref{eqs:hedimu}; whereas the structure of these blocks in~\(\cL\) are a consequence of the multivariate Taylor expansion~\eqref{eqs:trt} and its derivatives.

\paragraph{Local slow subspace}
The system~\eqref{eqs:odes} appears in the form of a `forced' linear system, so our first task is to understand the corresponding linear homogeneous system obtained by omitting the `forcing' (although here the the `forcing' is the uncertain coupling with neighbouring locales).
The corresponding homogeneous system is upper triangular, so its eigenvalues are those of~\(\fL_{0,0}\), namely $0$, $-1$ and~\(-3\) each with multiplicity six.
The twelve negative eigenvalues indicates that after transients decay, roughly like~\Ord{e^{-t}}, the system evolves on the 6D slow subspace of the six eigenvalues~\(0\).

Let's construct this 6D slow subspace.
Two eigenvectors corresponding to the zero eigenvalue are found immediately, namely
\begin{subequations}\label{eqs:bss6}%
\begin{align}&
\vv^{00}:=(1,0,\ldots,0),\\&
\vv^{01}=(0,-1,0,0,0,0,1,0\ldots,0).
\end{align}
A further four linearly independent vectors are generalised  eigenvectors:
\begin{align}&
\vv^{10}:=(0,0,-\tfrac29,1,0,\ldots,0),\\&
\vv^{20}:=(0,0,-\tfrac4{81},0,0,-\tfrac29,0,0,0,1,0,\ldots,0),\\&
\vv^{11}=(0,\tfrac89,0,0,-1,0,0,0,-\tfrac29,0,0,0,1,0,\ldots,0),\\&
\vv^{02}:=(0,0,\tfrac19,0,0,0,0,-1,0,0,0,0,0,0,0,1,0,0).
\end{align}
\end{subequations}
Setting the matrix $\cV :=\begin{bmatrix} \vv^{00}&\cdots&\vv^{02} \end{bmatrix}\in\RR^{18\times6}$, the slow subspace is then $\uv=\cV \uv_0$ where we conveniently choose to use $\uv_0:=(u_0^{00},\ldots,u_0^{02})\in\RR^6$ to directly parametrise the slow subspace because of the form chosen for the eigenvectors~$\vv^{mn}$.
On this slow subspace, from the eigenvectors via $\uv=\cV \uv_0$\,, the difference variables 
\begin{align*}&
\uv_1:=\big(
-u_0^{01}+\tfrac89u_0^{11}
,-u_0^{11}
,-u_0^{02}
,0,0,0\big),
\\&
\uv_2:=\big(
\tfrac19u_0^{02}-\tfrac29u_0^{10}-\tfrac4{81}u_0^{20}
,-\tfrac29u_0^{20}
,-\tfrac29u_0^{11}
,0,0,0\big).
\end{align*}
Further, on this slow subspace the evolution is determined by the upper triangular matrix in
\begin{equation}
\de t{\uv_0}=\bA\uv_0=\begin{bmatrix} 
0&-\frac13&0&\frac8{27}&0&\frac23\\ 
0&0&0&-\frac13&0&0\\ 
0&0&0&0&-\frac13&0\\ 
0&0&0&0&0&0\\ 
0&0&0&0&0&0\\ 
0&0&0&0&0&0\\ 
\end{bmatrix}\uv_0
\label{eq:A}
\end{equation}
The large number of zeros in the \ode\ system~\eqref{eq:A} reflects the zero eigenvalues, the zero-mean of the \(y\)~velocities for the three layers, and a useful upper triangular nature.
The non-zero elements in the first row give \(\dot u_0^{00}=
-\tfrac13u_0^{10}+\tfrac8{27}u_0^{20}+\tfrac23u_0^{02}\) which directly corresponds to the  macroscale model \pde~\eqref{eq:appdiff}.
The novelty of our approach is that we now go beyond such a basic model to quantitatively determine remainders---these remainders are  quantifiable errors in the model.

\subsection{Time dependent normal form}
\label{sec:tdnf}

Near identity coordinate transforms underpin modelling dynamics.
In particular, time-dependent coordinate transforms empower  understanding of the modelling of non-autonomous, and stochastic, dynamical systems \cite[e.g.]{Aulbach99, Arnold98, Roberts06k}.
This section analogously uses a time-dependent coordinate transformation to separate exactly the slow and fast modes of the system~\eqref{eqs:odes} in the presence of the `uncertain forcing' that couples the dynamics to neighbouring stations.
This is the first time the effects of such coupling have been quantified in multiscale problems with multiple large dimensions.

The coordinate transform introduces new dependent variables~\(\Uv(t)\).
In some sense, the new variables \(\uv\approx \Uv\) so the coordinate transform is `near identity'.
Let's choose to parametrise precisely the slow subspace of the system~\eqref{eqs:odes} by the variables~\(\Uv_0\): that is, on the subspace where the new stable variables \(\Uv_1=\Uv_2=\vec 0\),  we insist on the exact identity \(\uv=\Uv\).
This choice simplifies subsequent construction of slowly varying models such as~\eqref{eq:appdiff}.

In the coordinate transform, the effects of the uncertain forcings appear as integrals over their past history.
For any \(\mu>0\)\,, we define the convolution
\begin{equation}
\ou\big(w(t),tt,-\mu\big):=\int_0^te^{\mu(\tau-t)}w(\tau)\,d\tau\,.
\label{eq:con}
\end{equation}
Consequently,  the time derivative 
\(\de t{(\ou\big(w,tt,-\mu\big))}=-\mu\ou\big(w,tt,-\mu\big)+w\) which is a key property in the upcoming analysis.

Well established iteration described elsewhere \cite[e.g.]{Roberts06k} constructs the coordinate transformation.
The details are not significant here, all we need are the results.
The computer algebra code of \autoref{sec:camhe}, for the case \(N=2\), produces the following exact coordinate transform~\eqref{eqs:uU}: there is no neglect of any `small' terms.

\newcommand{\z}[2]{e^{\ifnum#2=-1 -\else#2\fi t}{\star}#1}
\newcommand{\parz}[1]{\parbox[t]{0.8\linewidth}{\raggedright\sloppy$#1$}}
\newif\ifyrog
\yrogtrue

Invoke the following time dependent, coordinate transform, \(\Uv\mapsto\uv\)\,: \ifyrog where the ellipses represent many terms we choose not to present so that you can more easily appreciate the overall structure,\else in full excruciating detail,\fi
\begin{subequations}\label{eqs:uU}%
\begin{align}
u^{00}_{0}={}&\parz{ \frac{8}{81} U^{20}_{2}+\frac{4}{9} U^{10}_{2}-\frac{2}{9} U
^{02}_{2}-\frac{16}{27} U^{11}_{1}+\frac{2}{3} U^{01}_{1}+U^{00}_{0}
,}\label{eq:ux000}\\
u^{10}_{0}={}&\parz{ \frac{4}{9} U^{20}_{2}+\frac{2}{3} U^{11}_{1}+U^{10}_{0}
,}\\
u^{01}_{0}={}&\parz{ \frac{4}{9} U^{11}_{2}+\frac{2}{3} U^{02}_{1}+U^{01}_{0}
,}\\
u^{20}_{0}={}&\parz{ U^{20}_{0}
,}\\
u^{11}_{0}={}&\parz{ U^{11}_{0}
,}\\
u^{02}_{0}={}&\parz{ U^{02}_{0}
,}\label{eq:ux002}\\
u^{00}_{1}={}&\parz{ U^{00}_{1}+\frac{8}{9} U^{11}_{0}-U^{01}_{0}+\z{u^{20}_{2y}
}{-1}+
{}\ifyrog\cdots\else
\frac{5}{27} \z{u^{20}_{2y}}{-3}+2 \z{u^{11}_{2x}}{-1}+\frac{
10}{27} \z{u^{11}_{2x}}{-3}+\frac{1}{4} \z{u^{02}_{2y}}{-1}-\frac{1
}{4} \z{u^{02}_{2y}}{-3}+\frac{13}{6} \z{u^{11}_{1y}}{-1}-\frac{53}{
54} \z{u^{11}_{1y}}{-3}+\frac{13}{12} \z{u^{02}_{1x}}{-1}-\frac{53}{
108} \z{u^{02}_{1x}}{-3}+\z{u^{20}_{0y}}{-1}-\frac{19}{27} \z{u^{20}
_{0y}}{-3}+2 \z{u^{11}_{0x}}{-1}-\frac{38}{27} \z{u^{11}_{0x}}{-3}
+\frac{1}{4} \z{u^{02}_{0y}}{-1}-\frac{1}{4} \z{u^{02}_{0y}}{-3}+
\frac{3}{2} \z{\z{u^{20}_{2y}}{-1}}{-1}+\frac{1}{18} \z{\z{u^{20}_{2
y}}{-3}}{-3}+3 \z{\z{u^{11}_{2x}}{-1}}{-1}+\frac{1}{9} \z{\z{u^{
11}_{2x}}{-3}}{-3}+\frac{3}{2} \z{\z{u^{02}_{2y}}{-1}}{-1}+2 \z{
\z{u^{11}_{1y}}{-1}}{-1}-\frac{1}{9} \z{\z{u^{11}_{1y}}{-3}}{-3}
+\z{\z{u^{02}_{1x}}{-1}}{-1}-\frac{1}{18} \z{\z{u^{02}_{1x}}{-3}
}{-3}-\frac{1}{9} \z{\z{u^{20}_{0y}}{-3}}{-3}-\frac{2}{9} \z{\z{u^{
11}_{0x}}{-3}}{-3}+\frac{3}{2} \z{\z{u^{02}_{0y}}{-1}}{-1}-\z{\z
{\z{u^{20}_{2y}}{-1}}{-1}}{-1}-2 \z{\z{\z{u^{11}_{2x}}{-1}}{-1
}}{-1}-3 \z{\z{\z{u^{20}_{1x}}{-1}}{-1}}{-1}-\z{\z{\z{u^{20}_{0y
}}{-1}}{-1}}{-1}
\fi{}
-2 \z{\z{\z{u^{11}_{0x}}{-1}}{-1}}{-1}
,}\\
u^{10}_{1}={}&\parz{ U^{10}_{1}-U^{11}_{0}-\frac{3}{2} \z{u^{20}_{2y}}{-1}+
{}\ifyrog\cdots\else
\frac{1}{6} \z{u^{20}_{2y}}{-3}-3 \z{u^{11}_{2x}}{-1}+\frac{1}{3} \z
{u^{11}_{2x}}{-3}-\z{u^{11}_{1y}}{-1}-\frac{1}{3} \z{u^{11}_{1y}}{
-3}-\frac{1}{2} \z{u^{02}_{1x}}{-1}-\frac{1}{6} \z{u^{02}_{1x}}{-3}-
\frac{1}{3} \z{u^{20}_{0y}}{-3}-\frac{2}{3} \z{u^{11}_{0x}}{-3}+\z{
\z{u^{20}_{2y}}{-1}}{-1}+2 \z{\z{u^{11}_{2x}}{-1}}{-1}+3 \z{\z{u
^{20}_{1x}}{-1}}{-1}+\z{\z{u^{20}_{0y}}{-1}}{-1}
\fi{}
+2 \z{\z{u^{11}_
{0x}}{-1}}{-1}
,}\\
u^{01}_{1}={}&\parz{ U^{01}_{1}-U^{02}_{0}-3 \z{u^{11}_{2y}}{-1}+
{}\ifyrog\cdots\else
\frac{1}{3} \z{
u^{11}_{2y}}{-3}-\frac{3}{2} \z{u^{02}_{2x}}{-1}+\frac{1}{6} \z{u^{0
2}_{2x}}{-3}-\frac{3}{2} \z{u^{02}_{1y}}{-1}-\frac{1}{2} \z{u^{02}_{
1y}}{-3}-\frac{2}{3} \z{u^{11}_{0y}}{-3}-\frac{1}{3} \z{u^{02}_{0x}
}{-3}+2 \z{\z{u^{11}_{2y}}{-1}}{-1}+\z{\z{u^{02}_{2x}}{-1}}{-1}
+\z{\z{u^{20}_{1y}}{-1}}{-1}+2 \z{\z{u^{11}_{1x}}{-1}}{-1}+2 \z{
\z{u^{11}_{0y}}{-1}}{-1}
\fi{}
+\z{\z{u^{02}_{0x}}{-1}}{-1}
,}\\
u^{20}_{1}={}&\parz{ U^{20}_{1}-\z{u^{20}_{2y}}{-1}-2 \z{u^{11}_{2x}}{-1}-3 \z
{u^{20}_{1x}}{-1}-\z{u^{20}_{0y}}{-1}-2 \z{u^{11}_{0x}}{-1}
,}\\
u^{11}_{1}={}&\parz{ U^{11}_{1}-2 \z{u^{11}_{2y}}{-1}-
{}\ifyrog\cdots\else
\z{u^{02}_{2x}}{-1}-\z{u
^{20}_{1y}}{-1}-2 \z{u^{11}_{1x}}{-1}-2 \z{u^{11}_{0y}}{-1}
\fi{}-\z{u^{
02}_{0x}}{-1}
,}\\
u^{02}_{1}={}&\parz{ U^{02}_{1}-3 \z{u^{02}_{2y}}{-1}-2 \z{u^{11}_{1y}}{-1}-\z
{u^{02}_{1x}}{-1}-3 \z{u^{02}_{0y}}{-1}
,}\\
u^{00}_{2}={}&\parz{ U^{00}_{2}-\frac{4}{81} U^{20}_{0}-\frac{2}{9} U^{10}_{0}+
\frac{1}{9} U^{02}_{0}-\frac{16}{81} \z{u^{20}_{2x}}{-3}+
{}\ifyrog\cdots\else
\frac{47}{54}
 \z{u^{11}_{2y}}{-1}-\frac{31}{54} \z{u^{11}_{2y}}{-3}+\frac{47}{108
} \z{u^{02}_{2x}}{-1}-\frac{31}{108} \z{u^{02}_{2x}}{-3}+\frac{5}{27
} \z{u^{20}_{1y}}{-1}-\frac{53}{243} \z{u^{20}_{1y}}{-3}+\frac{10}{
27} \z{u^{11}_{1x}}{-1}-\frac{106}{243} \z{u^{11}_{1x}}{-3}+\frac{1
}{4} \z{u^{02}_{1y}}{-1}-\frac{1}{36} \z{u^{02}_{1y}}{-3}-\frac{4}{
81} \z{u^{20}_{0x}}{-3}+\frac{10}{27} \z{u^{11}_{0y}}{-1}-\frac{8}{
27} \z{u^{11}_{0y}}{-3}+\frac{5}{27} \z{u^{02}_{0x}}{-1}-\frac{4}{27
} \z{u^{02}_{0x}}{-3}-\frac{16}{27} \z{\z{u^{20}_{2x}}{-3}}{-3}-
\frac{1}{3} \z{\z{u^{11}_{2y}}{-1}}{-1}-\frac{2}{9} \z{\z{u^{11}_{2y
}}{-3}}{-3}-\frac{1}{6} \z{\z{u^{02}_{2x}}{-1}}{-1}-\frac{1}{9} 
\z{\z{u^{02}_{2x}}{-3}}{-3}-\frac{1}{6} \z{\z{u^{20}_{1y}}{-1}}{
-1}-\frac{1}{162} \z{\z{u^{20}_{1y}}{-3}}{-3}-\frac{1}{3} \z{\z{u^{1
1}_{1x}}{-1}}{-1}-\frac{1}{81} \z{\z{u^{11}_{1x}}{-3}}{-3}+
\frac{1}{6} \z{\z{u^{02}_{1y}}{-3}}{-3}+\frac{14}{27} \z{\z{u^{20}_{
0x}}{-3}}{-3}-\frac{1}{3} \z{\z{u^{11}_{0y}}{-1}}{-1}+\frac{1}{9
} \z{\z{u^{11}_{0y}}{-3}}{-3}-\frac{1}{6} \z{\z{u^{02}_{0x}}{-1}
}{-1}+\frac{1}{18} \z{\z{u^{02}_{0x}}{-3}}{-3}+\frac{1}{9} \z{\z{\z
{u^{20}_{2x}}{-3}}{-3}}{-3}-\frac{1}{27} \z{\z{\z{u^{20}_{1y}}{
-3}}{-3}}{-3}-\frac{2}{27} \z{\z{\z{u^{11}_{1x}}{-3}}{-3}}{-3}
\fi{}
-\frac{2}{9} \z{\z{\z{u^{20}_{0x}}{-3}}{-3}}{-3}
,}\\
u^{10}_{2}={}&\parz{ U^{10}_{2}-\frac{2}{9} U^{20}_{0}-\frac{8}{9} \z{u^{20}_{2x}
}{-3}
{}\ifyrog\cdots\else
+\frac{1}{3} \z{u^{11}_{2y}}{-1}-\frac{1}{3} \z{u^{11}_{2y}}{
-3}+\frac{1}{6} \z{u^{02}_{2x}}{-1}-\frac{1}{6} \z{u^{02}_{2x}}{-3}+
\frac{1}{6} \z{u^{20}_{1y}}{-1}-\frac{17}{54} \z{u^{20}_{1y}}{-3}+
\frac{1}{3} \z{u^{11}_{1x}}{-1}-\frac{17}{27} \z{u^{11}_{1x}}{-3}-
\frac{2}{9} \z{u^{20}_{0x}}{-3}+\frac{1}{3} \z{u^{11}_{0y}}{-1}-
\frac{1}{3} \z{u^{11}_{0y}}{-3}+\frac{1}{6} \z{u^{02}_{0x}}{-1}-
\frac{1}{6} \z{u^{02}_{0x}}{-3}+\frac{1}{3} \z{\z{u^{20}_{2x}}{-3}
}{-3}-\frac{1}{9} \z{\z{u^{20}_{1y}}{-3}}{-3}-\frac{2}{9} \z{\z{u^{
11}_{1x}}{-3}}{-3}
\fi{}
-\frac{2}{3} \z{\z{u^{20}_{0x}}{-3}}{-3}
,}\\
u^{01}_{2}={}&\parz{ U^{01}_{2}-\frac{2}{9} U^{11}_{0}-\frac{8}{27} \z{u^{20}_{2y}
}{-3}
{}\ifyrog\cdots\else
-\frac{16}{27} \z{u^{11}_{2x}}{-3}+\frac{1}{2} \z{u^{02}_{2y}
}{-1}-\frac{1}{2} \z{u^{02}_{2y}}{-3}+\frac{1}{3} \z{u^{11}_{1y}}{
-1}-\frac{17}{27} \z{u^{11}_{1y}}{-3}+\frac{1}{6} \z{u^{02}_{1x}}{-1
}-\frac{17}{54} \z{u^{02}_{1x}}{-3}-\frac{2}{27} \z{u^{20}_{0y}}{-3}
-\frac{4}{27} \z{u^{11}_{0x}}{-3}+\frac{1}{2} \z{u^{02}_{0y}}{-1}-
\frac{1}{2} \z{u^{02}_{0y}}{-3}+\frac{1}{9} \z{\z{u^{20}_{2y}}{-3}
}{-3}+\frac{2}{9} \z{\z{u^{11}_{2x}}{-3}}{-3}-\frac{2}{9} \z{\z{u^{
11}_{1y}}{-3}}{-3}-\frac{1}{9} \z{\z{u^{02}_{1x}}{-3}}{-3}-
\frac{2}{9} \z{\z{u^{20}_{0y}}{-3}}{-3}
\fi{}
-\frac{4}{9} \z{\z{u^{11}_{0x
}}{-3}}{-3}
,}\\
u^{20}_{2}={}&\parz{ U^{20}_{2}+\z{u^{20}_{2x}}{-3}-\frac{1}{3} \z{u^{20}_{1y}
}{-3}-\frac{2}{3} \z{u^{11}_{1x}}{-3}-2 \z{u^{20}_{0x}}{-3}
,}\\
u^{11}_{2}={}&\parz{ U^{11}_{2}+\frac{1}{3} \z{u^{20}_{2y}}{-3}+
{}\ifyrog\cdots\else
\frac{2}{3} \z{u
^{11}_{2x}}{-3}-\frac{2}{3} \z{u^{11}_{1y}}{-3}-\frac{1}{3} \z{u^{02
}_{1x}}{-3}-\frac{2}{3} \z{u^{20}_{0y}}{-3}
\fi{}
-\frac{4}{3} \z{u^{11}_{0
x}}{-3}
,}\\
u^{02}_{2}={}&\parz{ U^{02}_{2}+\frac{2}{3} \z{u^{11}_{2y}}{-3}+
{}\ifyrog\cdots\else
\frac{1}{3} \z{u
^{02}_{2x}}{-3}-\z{u^{02}_{1y}}{-3}-\frac{4}{3} \z{u^{11}_{0y}}{-3
}
\fi{}-\frac{2}{3} \z{u^{02}_{0x}}{-3}
.}
\end{align}
\end{subequations}
In these new variables~\(\Uv\) the original system~\eqref{eqs:odes} is exactly the separated system\ifyrog\else, in glorious detail,\fi
\begin{subequations}\label{eqs:dUdt}%
\begin{align}
\dot U^{00}_{0}={}&\parz{ \frac{8}{27} U^{20}_{0}-\frac{1}{3} U^{10}_{0}+\frac{2}{
3} U^{02}_{0}+\frac{32}{27} \z{u^{20}_{2x}}{-3}
{}\ifyrog\cdots\else
+\frac{14}{9} \z{u^{11}
_{2y}}{-1}+\frac{2}{9} \z{u^{11}_{2y}}{-3}+\frac{7}{9} \z{u^{02}_{2x
}}{-1}+\frac{1}{9} \z{u^{02}_{2x}}{-3}-\frac{2}{9} \z{u^{20}_{1y}}{
-1}+\frac{34}{81} \z{u^{20}_{1y}}{-3}-\frac{4}{9} \z{u^{11}_{1x}}{
-1}+\frac{68}{81} \z{u^{11}_{1x}}{-3}+\z{u^{02}_{1y}}{-1}+\frac{1}{3
} \z{u^{02}_{1y}}{-3}+\frac{8}{27} \z{u^{20}_{0x}}{-3}-\frac{4}{9} 
\z{u^{11}_{0y}}{-1}+\frac{8}{9} \z{u^{11}_{0y}}{-3}-\frac{2}{9} \z{u
^{02}_{0x}}{-1}+\frac{4}{9} \z{u^{02}_{0x}}{-3}-\frac{4}{9} \z{\z{u
^{20}_{2x}}{-3}}{-3}-\frac{4}{3} \z{\z{u^{11}_{2y}}{-1}}{-1}-
\frac{2}{3} \z{\z{u^{02}_{2x}}{-1}}{-1}-\frac{2}{3} \z{\z{u^{20}_{1y
}}{-1}}{-1}+\frac{4}{27} \z{\z{u^{20}_{1y}}{-3}}{-3}-\frac{4}{3}
 \z{\z{u^{11}_{1x}}{-1}}{-1}+\frac{8}{27} \z{\z{u^{11}_{1x}}{-3}
}{-3}+\frac{8}{9} \z{\z{u^{20}_{0x}}{-3}}{-3}-\frac{4}{3} \z{\z{u^{
11}_{0y}}{-1}}{-1}
\fi{}
-\frac{2}{3} \z{\z{u^{02}_{0x}}{-1}}{-1}
,}\label{eq:dU000dt}\\
\dot U^{10}_{0}={}&\parz{ -\frac{1}{3} U^{20}_{0}-\frac{4}{3} \z{u^{20}_{2x}}{-3
}
{}\ifyrog\cdots\else
+\frac{4}{3} \z{u^{11}_{2y}}{-1}+\frac{2}{3} \z{u^{02}_{2x}}{-1}+
\frac{2}{3} \z{u^{20}_{1y}}{-1}+\frac{4}{9} \z{u^{20}_{1y}}{-3}+
\frac{4}{3} \z{u^{11}_{1x}}{-1}+\frac{8}{9} \z{u^{11}_{1x}}{-3}+
\frac{8}{3} \z{u^{20}_{0x}}{-3}+\frac{4}{3} \z{u^{11}_{0y}}{-1}
\fi{}
+\frac{2}{3} \z{u^{02}_{0x}}{-1}
,}\\
\dot U^{01}_{0}={}&\parz{ -\frac{1}{3} U^{11}_{0}-\frac{4}{9} \z{u^{20}_{2y}}{-3
}-
{}\ifyrog\cdots\else
\frac{8}{9} \z{u^{11}_{2x}}{-3}+2 \z{u^{02}_{2y}}{-1}+\frac{4}{3} 
\z{u^{11}_{1y}}{-1}+\frac{8}{9} \z{u^{11}_{1y}}{-3}+\frac{2}{3} \z{u
^{02}_{1x}}{-1}+\frac{4}{9} \z{u^{02}_{1x}}{-3}+\frac{8}{9} \z{u^{20
}_{0y}}{-3}+\frac{16}{9} \z{u^{11}_{0x}}{-3}
\fi{}
+2 \z{u^{02}_{0y}}{-1}
,}\\
\dot U^{20}_{0}={}&\parz{ -4 u^{20}_{2x}-\frac{2}{3} u^{20}_{1y}-\frac{4}{3} u^{11
}_{1x}-u^{20}_{0x}
,}\\
\dot U^{11}_{0}={}&\parz{ -\frac{4}{3} u^{20}_{2y}-\frac{8}{3} u^{11}_{2x}-\frac{4
}{3} u^{11}_{1y}-\frac{2}{3} u^{02}_{1x}-\frac{1}{3} u^{20}_{0y}-\frac{2
}{3} u^{11}_{0x}
,}\\
\dot U^{02}_{0}={}&\parz{ -\frac{8}{3} u^{11}_{2y}-\frac{4}{3} u^{02}_{2x}-2 u^{02
}_{1y}-\frac{2}{3} u^{11}_{0y}-\frac{1}{3} u^{02}_{0x}
,}\\
\dot U^{00}_{1}={}&\parz{ -\frac{4}{9} U^{11}_{2}-U^{01}_{2}-U^{10}_{1}-\frac{2}{3
} U^{02}_{1}-U^{00}_{1}
,}\label{eq:dU100dt}\\
\dot U^{10}_{1}={}&\parz{ -U^{11}_{2}-U^{20}_{1}-U^{10}_{1}
,}\\
\dot U^{01}_{1}={}&\parz{ -U^{02}_{2}-U^{11}_{1}-U^{01}_{1}
,}\\
\dot U^{20}_{1}={}&\parz{ -U^{20}_{1}
,}\\
\dot U^{11}_{1}={}&\parz{ -U^{11}_{1}
,}\\
\dot U^{02}_{1}={}&\parz{ -U^{02}_{1}
,}\\
\dot U^{00}_{2}={}&\parz{ -\frac{8}{27} U^{20}_{2}+\frac{1}{3} U^{10}_{2}-3 U^{00}
_{2}-\frac{4}{9} U^{11}_{1}-\frac{1}{3} U^{01}_{1}
,}\\
\dot U^{10}_{2}={}&\parz{ \frac{1}{3} U^{20}_{2}-3 U^{10}_{2}-\frac{1}{3} U^{11}_{
1}
,}\\
\dot U^{01}_{2}={}&\parz{ \frac{1}{3} U^{11}_{2}-3 U^{01}_{2}-\frac{1}{3} U^{02}_{
1}
,}\\
\dot U^{20}_{2}={}&\parz{ -3 U^{20}_{2}
,}\\
\dot U^{11}_{2}={}&\parz{ -3 U^{11}_{2}
,}\\
\dot U^{02}_{2}={}&\parz{ -3 U^{02}_{2}
.}
\label{eq:dU202dt}
\end{align}
\end{subequations}
In these new variables, the \ode{}s~\eqref{eq:dU100dt}--\eqref{eq:dU202dt} in this separated system immediately show that all the new stable variables \(U_1^{mn}, U_2^{mn}\to0\) as time \(t\to\infty\)\,.
Moreover, they decay exponentially quickly, \(U_1^{mn}, U_2^{mn}=\Ord{e^{-\gamma t}}\) for any chosen rate \(0<\gamma<1\).
That is, \(\Uv_1=\Uv_2=\vec 0\) is the exact slow subspace for the `forced' system~\eqref{eqs:odes}.

Further, and usefully, by the construction of~\eqref{eq:ux000}--\eqref{eq:ux002}, on this slow subspace \(\uv_0^{mn}=\Uv_0^{mn}\), exactly.

\subsection{The slowly varying model}
Recall the exact Taylor multinomial~\eqref{eqs:trt}.
Given the exact coordinate transform~\eqref{eqs:uU}, and that \(U_1^{mn}, U_2^{mn}=\Ord{e^{-\gamma t}}\), the Taylor multinomial~\eqref{eqs:trt} establishes that, based upon the station~\((X,Y)\), the mean field
\begin{eqnarray}
u_0(x,y,t)&=&u_0^{00}(X,Y,t)+u_0^{10}(X,Y,t)\pxt1 
+u_0^{01}(X,Y,t)\pyt1
\nonumber\\&&{}
+u_0^{20}(X,Y,x,y,t)\pxt2
+u_0^{11}(X,Y,x,y,t)\pxt1\pyt1
\nonumber\\&&{}
 +u_0^{02}(X,Y,x,y,t)\pyt2
+\Ord{e^{-\gamma t}}.
\label{eq:slowcid}
\end{eqnarray}
Crucially, the left-hand side is independent of the station~\((X,Y)\).
If the right-hand side was just a local approximation, then the field it generates would depend upon the station~\((X,Y)\).
But the right-hand side is exact (with its unknown but exponentially quickly decaying transients).
This exactness is maintained because we keep the remainder terms in the analysis.
Consequently, the mean field given by expression~\eqref{eq:slowcid} is independent of the station~\((X,Y)\) despite \((X,Y)\)~appearing on the right-hand side.

To obtain an exact \pde\ of the slow variations in the mean field~\(u_0\), take the time derivative of~\eqref{eq:slowcid} and evaluate at \((x,y)=(X,Y)\).  
Remembering the derivative of the history convolution, that \(\de t{}(\ou\big(w,tt,-\mu\big))=-\mu\ou\big(w,tt,-\big)+w\), the \ode~\eqref{eq:dU000dt} implies 
\begin{eqnarray*}
\D t{u_0}&=&
\D t{U_0^{00}}+\Ord{e^{-\gamma t}}
\\&=&\parz{-\frac{1}{3} U^{10}_{0} +\frac{8}{27} U^{20}_{0}+\frac{2}{
3} U^{02}_{0}+\frac{32}{27} \z{u^{20}_{2x}}{-3} +
{}\ifyrog\cdots\else
\frac{14}{9} \z{u^{11}
_{2y}}{-1}+\frac{2}{9} \z{u^{11}_{2y}}{-3}+\frac{7}{9} \z{u^{02}_{2x
}}{-1}+\frac{1}{9} \z{u^{02}_{2x}}{-3}-\frac{2}{9} \z{u^{20}_{1y}}{
-1}+\frac{34}{81} \z{u^{20}_{1y}}{-3}-\frac{4}{9} \z{u^{11}_{1x}}{
-1}+\frac{68}{81} \z{u^{11}_{1x}}{-3}+\z{u^{02}_{1y}}{-1}+\frac{1}{3
} \z{u^{02}_{1y}}{-3}+\frac{8}{27} \z{u^{20}_{0x}}{-3}-\frac{4}{9} 
\z{u^{11}_{0y}}{-1}+\frac{8}{9} \z{u^{11}_{0y}}{-3}-\frac{2}{9} \z{u
^{02}_{0x}}{-1}+\frac{4}{9} \z{u^{02}_{0x}}{-3}-\frac{4}{9} \z{\z{u
^{20}_{2x}}{-3}}{-3}-\frac{4}{3} \z{\z{u^{11}_{2y}}{-1}}{-1}-
\frac{2}{3} \z{\z{u^{02}_{2x}}{-1}}{-1}-\frac{2}{3} \z{\z{u^{20}_{1y
}}{-1}}{-1}+\frac{4}{27} \z{\z{u^{20}_{1y}}{-3}}{-3}-\frac{4}{3}
 \z{\z{u^{11}_{1x}}{-1}}{-1}+\frac{8}{27} \z{\z{u^{11}_{1x}}{-3}
}{-3}+\frac{8}{9} \z{\z{u^{20}_{0x}}{-3}}{-3}-\frac{4}{3} \z{\z{u^{
11}_{0y}}{-1}}{-1}
\fi{}
-\frac{2}{3} \z{\z{u^{02}_{0x}}{-1}}{-1}
+\Ord{e^{-\gamma t}}}
\\&=&\parz{-\frac{1}{3} u^{10}_{0} +\frac{8}{27} u^{20}_{0} +\frac{2}{
3} u^{02}_{0}
+\z{\big(
\frac{14}{9} u^{11}_{2y}
+\frac{7}{9} u^{02}_{2x}
-\frac{2}{9} u^{20}_{1y}
-\frac{4}{9} u^{11}_{1x}
+u^{02}_{1y}
-\frac{4}{9} u^{11}_{0y}
-\frac{2}{9} u^{02}_{0x}
\big)}{-1}
+\z{\big(
\frac{32}{27} u^{20}_{2x}
+\frac{2}{9} u^{11}_{2y}
+\frac{1}{9} u^{02}_{2x}
+\frac{34}{81} u^{20}_{1y}
+\frac{68}{81} u^{11}_{1x}
+\frac{1}{3} u^{02}_{1y}
+\frac{8}{27} u^{20}_{0x}
+\frac{8}{9} u^{11}_{0y}
+\frac{4}{9} u^{02}_{0x}
\big)}{-3}
+\z{\z{\big(
-\frac{4}{3} u^{11}_{2y}
-\frac{2}{3} u^{02}_{2x}
-\frac{2}{3} u^{20}_{1y}
-\frac{4}{3} u^{11}_{1x}
-\frac{4}{3} u^{11}_{0y}
-\frac{2}{3} u^{02}_{0x}
\big)}{-1}}{-1}
+\z{\z{\big(
-\frac{4}{9} u^{20}_{2x}
+\frac{4}{27} u^{20}_{1y}
+\frac{8}{27} u^{11}_{1x}
+\frac{8}{9} u^{20}_{0x}
\big)}{-3}}{-3}
+\Ord{e^{-\gamma t}}}
\\&=&-\frac{1}{3} \D x{u_0}+\frac{8}{27} \DD x{u_0}+\frac{2}{
3}\DD y{u_0}
+r
+\Ord{e^{-\gamma t}},
\end{eqnarray*}
for remainder
\begin{eqnarray*}
r&:=&\parz{ \z{\big(
\frac{14}{9} u^{11}_{2y}
+\frac{7}{9} u^{02}_{2x}
-\frac{2}{9} u^{20}_{1y}
-\frac{4}{9} u^{11}_{1x}
+u^{02}_{1y}
-\frac{4}{9} u^{11}_{0y}
-\frac{2}{9} u^{02}_{0x}
\big)}{-1}
+\z{\big(
\frac{32}{27} u^{20}_{2x}
+\frac{2}{9} u^{11}_{2y}
+\frac{1}{9} u^{02}_{2x}
+\frac{34}{81} u^{20}_{1y}
+\frac{68}{81} u^{11}_{1x}
+\frac{1}{3} u^{02}_{1y}
+\frac{8}{27} u^{20}_{0x}
+\frac{8}{9} u^{11}_{0y}
+\frac{4}{9} u^{02}_{0x}
\big)}{-3}
+\z{\z{\big(
-\frac{4}{3} u^{11}_{2y}
-\frac{2}{3} u^{02}_{2x}
-\frac{2}{3} u^{20}_{1y}
-\frac{4}{3} u^{11}_{1x}
-\frac{4}{3} u^{11}_{0y}
-\frac{2}{3} u^{02}_{0x}
\big)}{-1}}{-1}
+\z{\z{\big(
-\frac{4}{9} u^{20}_{2x}
+\frac{4}{27} u^{20}_{1y}
+\frac{8}{27} u^{11}_{1x}
+\frac{8}{9} u^{20}_{0x}
\big)}{-3}}{-3},}
\end{eqnarray*}
that couples the local dynamics to its neighbouring locales.
Consequently, an exact statement of the mean field~\(u_0\) is thus
\begin{equation}
\D t{u_0}=
-\frac{1}{3} \D x{u_0}+\frac{8}{27} \DD x{u_0}+\frac{2}{
3}\DD y{u_0}
+r
+\Ord{e^{-\gamma t}}.\label{eq:pdeC}
\end{equation}
In principle, equation~\eqref{eq:pdeC} is an exact integro-differential equation for the system: the integral part coming from the history convolutions hidden within the coupling~\(r\) with other locales.
The macroscale approximation is to neglect both the transients and the integral coupling.
The rigorous, macroscale, slowly varying, model is then the Fokker--Planck\slash advection-diffusion \pde~\eqref{eq:appdiff} obtained from~\eqref{eq:pdeC} with \(\Ord{e^{-\gamma t}}\) neglected as a quickly decaying transient, and the uncertain coupling~$r$ neglected as its error.

Since the remainder coupling~\(r\) is a linear combination of the subtle spatial derivatives~\eqref{eq:uxyd} of second derivatives, here the magnitude of the neglected coupling is that of third order spatial derivatives.

The analysis detailed in this section is for the chosen truncation \(N=2\) in the multivariate Taylor series~\eqref{eqs:trt}.
Alternatively, upon choosing truncation \(N=3\) the computer algebra of \autoref{sec:camhe} derives the higher-order mean-field model%
\footnote{In some disciplines this \pde\ would be viewed as providing the next term in a Kramers--Moyal expansion of the mean random walk \pdf\ \cite[e.g.]{Pawula67}.}
\begin{equation*}
\D t{u_0}=
-\frac{1}{3} \D x{u_0}+\frac{8}{27} \DD x{u_0}+\frac{2}{
3}\DD y{u_0} +\frac{16}{243}\Dn x3{u_0} -\frac{20}{27}\frac{\partial^3u_0}{\partial x\partial y^2}
+r+\Ord{e^{-\gamma t}}
\end{equation*}
for some even more voluminous coupling remainder~\(r\).
Both such \pde{}s are examples of so-called mixed order models, as are the other \pde{}s derived with larger truncation~\(N\).
The extant mathematical methodologies of homogenisation and multiple scales promote an aversion to such mixed order models.
But our approach rigorously justifies such \pde{}s as mathematical models with the derived remainder~\(r\) being the quantifiable error.

\section{A PDE models interior plate dynamics}
\label{sec:pmid}

Inspired by the successful exact modelling of the random walker (heat exchanger) in \autoref{sec:mdhx}, this section establishes analogous exact modelling in the very wide class~\eqref{eq:upde} of linear systems of \pde{}s in multiple space dimensions.
\autoref{fig:schema} schematically overviews this section.

\begin{figure}
\centering
\tikzsetnextfilename{schema}
\tikzexternaldisable
\begin{tikzpicture}[
linearrow/.style = { draw, thick, -latex },
line/.style = { draw, thick},
method/.style = {diamond, draw=red, thick, fill=red!10,text centered,
 inner sep=1pt , rounded corners, aspect=6},
block/.style = { rectangle, draw=blue, thick, 
 fill=blue!10, text centered,
 rounded corners, minimum height=2em },
]
\matrix [column sep=5mm, row sep=5mm] {
\node[block] (u) {Original \pde~\eqref{eq:upde} for $u(\xv,y,t)$, $u\in\UU$ };\\
\node[method] (L) {LagrangeÕs Remainder Theorem \eqref{eqs:utaylor}};\\
\node[block] (un) {$m\tN$ \des~\eqref{eq:odesn} for $u^{(\nv)}(\Xv,y,t)$, $u^{(\nv)}\in\UU$};\\
 \node[method] (G) {Generating function $\cG$ \eqref{eq:Gn}};\\
\node[block] (tu) {local \pde~\eqref{eq:updeng} for $\tu(\xiv,\Xv,y,t)$, $\tu\in\UU_N$};\\
 \node[method] (LA) {Linear algebra on $\UU_N$, \S\S\ref{sec:mlas}--\ref{sec:hruf}};\\
\node[block] (c) {\(m\) \pde{}s~\eqref{eq:dcdta} for $U(\xv,t)$ with closure error~\eqref{eq:pderemain}};\\
 };
\begin{scope} [every path/.style=linearrow]
 \path (L) -- (un);
 \path (G)--(tu);
 \path (LA)--(c);
\end{scope}
\begin{scope} [every path/.style=line]
 \path (u) -- (L);
 \path (un)--(G);
 \path (tu)--(LA);
\end{scope}
\end{tikzpicture}
\caption{Flow chart describing the modelling scheme, from the original \pde\ for $u(\xv,y,t)$ to the macroscale model \pde\ for  the emergent `mean field' variables~$U(\xv,t)$.
Blue rectangles describe different stages of the modelling and red diamonds describe the method for obtaining one model from another.
}
\label{fig:schema}
\end{figure}
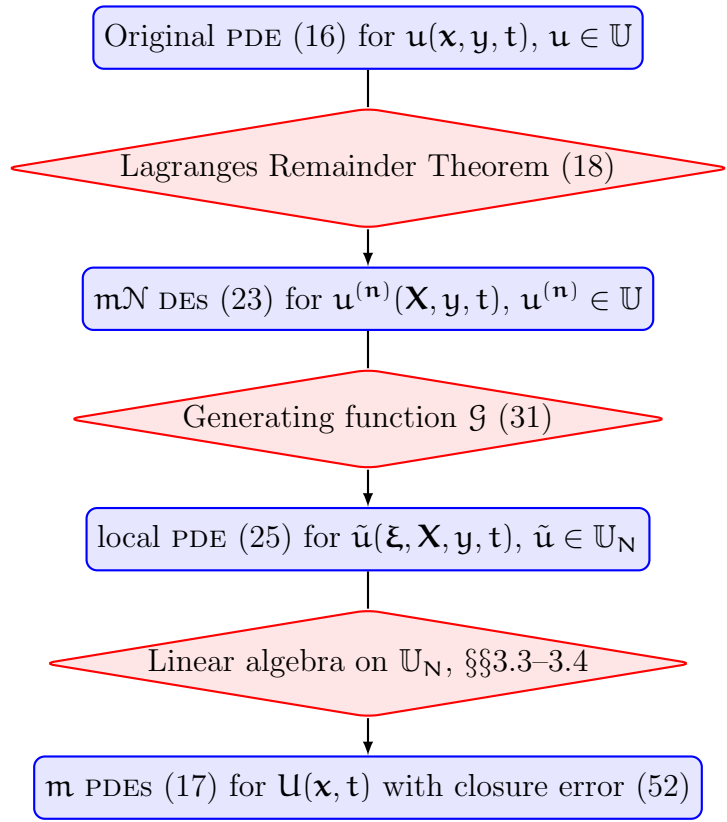

This section develops a rigorous approach to the modelling of fields in multiscale domains~\(\XX\times\YY\)\,.  
We suppose that \XX\ is an open domain in~\(\RR^M\) of large `macroscale' extent, called the plate, whereas \YY\ is a relatively small `microscale' cross-sectional domain in some Hilbert space.%
\footnote{For example, in the case of the modelling of an elastic plate, \YY~represents the thickness of the plate, and \XX~represents the relatively large width of the plate.
Alternatively,  if we were analysing the probability density function of some stochastic system, then space~\YY\ could be unbounded so long as there exists a quasi-stationary distribution in~\YY\ with operator~\(\fL_\ov\) having a suitable spectral gap \cite[\S21.2, e.g.]{Roberts2014a}.}
We consider the dynamics of some field~\(u(\xv,y,t)\) in a given Hilbert space~\UU\ (finite or infinite dimensional), where \(u:\XX\times\YY\times\RR\to\UU\) is a function of \(M\)-dimensional position \(\xv\in\XX\subset \RR^M\), cross-sectional position \(y\in\YY\), and time \(t\in\RR\) (in \S\ref{sec:mdhx} variable~\(y\) denoted a `large' dimension variable which in this general theory are all gathered into~\xv, whereas hereafter \(y\)~denotes position in the `thin' dimensions). 
We suppose the field~\(u(\xv,y,t)\) satisfies a specified \pde\ in the linear class 
\begin{equation}
\D tu=\fL_\ov u+\fL_{(1,0,\ldots,0)}\D{x_1}u
+\cdots+\fL_{(0,\ldots,0,1)}\D{x_M}u+\cdots
=\sum_{|\kv|=0}^{\infty} \fL_{\kv} \partial_{\xv}^{\kv} u \,,
\label{eq:upde}
\end{equation}
where the \(M\)-dimensional (mixed) derivative of order~\(|\kv|=k_1+k_2+\cdots+k_M\) is
\begin{equation*}
\partial_{\xv}^{\kv}=\frac{\partial^{|\kv|}}
{\partial x_1^{k_1}\partial x_2^{k_2}\cdots\partial x_M^{k_M}}\,,
\end{equation*}
and where the infinite sum is notionally written as being over all possible multi-indices~$\kv\in\NN_0^M$ (as usual, the set of natural numbers \(\NN_0:=\{0,1,2,\ldots\}\)). 
However, in application the sum of terms in the \pde~\eqref{eq:upde} typically truncate at the first or second order derivatives (as in the \pde~\eqref{eqs:hedim}). 
Nonetheless, our analysis caters for arbitrarily high order \pde{}s.
Such sums truncate as we do assume that only a finite number of operators~\(\cL_{\kv}\) are non-zero.  
Thus, such `infinite' sums truncate at some definite but unspecified order of~$|\kv|$.

The plate domain~\(\XX\) (open) may be finite, or infinite, or multi-periodic.
The plate domain~\XX\  generally excludes boundary layers and any internal `shocks' or `cracks'---it need not even be connected.
In application, the domain~\XX\ will be chosen a finite distance from boundary layers or shocks so that boundary layer structures have decayed enough for the remainder coupling effects (e.g.,~\eqref{eq:pdeC}) to be below some chosen threshold error on~\(\partial\XX\).
The `microscale' cross-section~\YY\ may be as simple as the index set~\(\{1,2,3\}\) as for the random walker\slash heat exchanger~\eqref{sec:mdhx}, or the whole of~\(\RR^\ydim\).
The operators~\(\fL_{\kv}\) are assumed autonomous and independent of in-plate position~\(\xv\); they only operate in the `microscale' cross-section~\YY.
Many problems which at first appear inhomogeneous in~\xv\ may be embedded into an ensemble of phase-shifted problems such that the resultant ensemble has operators~\(\fL_{\kv}\) independent of  position~\xv---see the example of \autoref{sec:ehmd2d}.

For \pde{}s in the general form~\eqref{eq:upde}, assume the field~\(u\) is smooth enough to have continuous \(2N\)~derivatives in~\(\xv\), \(u\in C^{2N}(\XX\times\YY\times\RR,\UU)\), for some pre-specified Taylor series truncation~\(N\). 
The choice of truncation~$N$ is only constrained by the differentiability class of field~$u$, and not by any truncation of the infinite~$k$ sum in~\eqref{eq:upde}.

This section, through to the very end of \autoref{sec:hruf}, establishes the following proposition.

\begin{proposition}[slowly varying \pde] \label{thm:svpde}
Let \(u(\xv,y,t)\) be governed by a linear \pde\ of the form~\eqref{eq:upde} satisfying Assumption~\ref{ass:spec}.  
Define the `macroscale field' \(U(\xv,t):=\langle Z^\ov(y),u(\xv,y,t)\rangle\), \(U:\XX\times\RR\to\RR^m\), for both~\(Z^\ov(y)\) and inner product of Definition~\ref{def:baseig}.
Then, in the regime of `slowly varying solutions' the macroscale field~\(U\) satisfies the \pde
\begin{equation}
\D tU=\sum_{|\nv|=0}^NA_{\nv}\partial_{\xv}^{\nv}U\,, \quad \xv\in\XX\,,
\label{eq:dcdta}
\end{equation}
in terms of \(m\times m\) matrices~\(A_{\nv}\) given by~\eqref{eq:toeplam}, 
to a closure error quantified by~\eqref{eq:pderemain}, and 
upon neglecting transients decaying exponentially quickly in time.
\end{proposition}

To cater for the intricacies of this problem we use a lot of symbols. \autoref{sec:nota} summarises the notation for convenient reference.

\subsection{Rewrite the local field}
\label{sec:rlf}

Choose a cross-section at plate station $\Xv\in\XX$\,, as shown schematically in \autoref{fig:hestation}.
Then, generalising~\eqref{eqs:utaylor1}, invoke the multivariate Lagrange's Remainder Theorem to write the field~\(u\) in terms of an \(N\)th~degree local Taylor multinomial about the cross-section \(\xv=\Xv\):
\begin{subequations}\label{eqs:utaylor}%
\begin{equation}
u(\xv,y,t)=\sum_{|\nv|=0}^{N-1}u^{(\nv)}(\Xv,y,t)\pxvt{\nv}
+\sum_{|\nv|=N}u^{(\nv)}(\Xv,\xv,y,t)\pxvt{\nv}\,,
\label{eq:utrtn}
\end{equation}
with the multi-index factorial \(\nv!:=n_1!n_2!\cdots n_M!\)\,,  the multi-index power \(\xv^{\nv}:=x_1^{n_1}x_2^{n_2}\cdots x_M^{n_M}\),
and where
\begin{itemize}
\item in the first sum, for \(|\nv|<N\)\,,
\begin{equation}
u^{(\nv)}(\Xv,y,t):=\partial_{\xv}^{\nv}u\big|_{\xv=\Xv}\,,
\label{eq:nterm}
\end{equation}
and $u^{(\nv)}:\XX\times\YY\times\RR\to\UU$\,.
\item whereas for the cases in the second sum of \(|\nv|=N\)\,, by Lagrange's Remainder Theorem for multivariate Taylor series~\cite[e.g.,][Eq.~(1.27)]{Taylor2011},  
\begin{equation}
u^{(\nv)}(\Xv,\xv,y,t):=N\int_0^1(1-s)^{N-1}\partial_{\xv}^{\nv}u\big|_{\Xv+s(\xv-\Xv)}\,ds\,, 
\label{eq:Nintterm}
\end{equation}
and $u^{(\nv)}:\XX\times\XX\times\YY\times\RR\to\UU$\,.
\end{itemize}
\end{subequations}
As the domain~\XX\ is open, Lagrange's Remainder Theorem~\eqref{eqs:utaylor} holds for all \(\xv\in\chi(\Xv)\subseteq\XX\) where \(\chi(\Xv)\)~is any open subset of~\XX\ such that for all points \(\xv\in\chi(\Xv)\) the convex combination \(\Xv+s(\xv-\Xv)\in\chi(\Xv)\) for every \(0\leq s\leq1\)\,.
The superscript notation~\(u^{(\nv)}\) reflects that the functions~\eqref{eq:nterm}--\eqref{eq:Nintterm} are fundamentally derivatives of the field~\(u\), even though they appear as independent variables in the analysis of this section (and as first posited by \cite{Roberts88a}).
Although at the highest order \(|\nv|=N\) the coefficient functions \(u^{(\nv)}(\Xv,\xv,y,t)\) in principle are well defined, in our macroscale modelling the spatial derivatives of these~\(u^{(\nv)}\) appear as an uncertain part of the modelling closure.  
The reason for the uncertainty is that, at the highest order \(|\nv|=N\)\,, the derivatives of these~\(u^{(\nv)}\) depend upon the field~\(u(\xv,y,t)\) between the station~\Xv\ and a point of interest~\xv.

\paragraph{Derive exact local DEs}
Generalising~\eqref{eqs:5pde}, let's derive some exact \des\ (either \ode{}s or \pde{}s depending upon~\YY) for the the evolution of the coefficients~\(u^{(\nv)}(\Xv,y,t)\) and~\(u^{(\nv)}(\Xv,\xv,y,t)\).
The specified \pde~\eqref{eq:upde} invokes various spatial derivatives of the field~\(u\).
Since the domain~\XX\ is open and the multivariate Taylor multinomial~\eqref{eq:utrtn}  applies in the open neighbourhood~\(\chi(\Xv)\) of each station~\Xv, and \(u\in C^{2N}\), the Taylor multinomial can be differentiated \(\ellv\)~times, \(|\ellv|\leq N\)\,.  
The multivariate Taylor multinomial~\eqref{eq:utrtn} gives, after some rearrangement
\begin{align}
\partial_{\xv}^{\ellv} u={}&{}\sum_{|\nv|=0}^{N-|\ellv|-1}u^{(\nv+\ellv)}\pxvt {\nv}
\nonumber\\{}&{}
+\sum_{|\nv|=N}\sum_{\mv=(\nv-\ellv)^\oplus}^{\nv}\binom{\ellv}{\nv-\mv}\partial_{\xv}^{\mv+\ellv-\nv}u^{(\nv)}\pxvt {\mv}\,,
\label{eq:dudxl}
\end{align}
where in the range of some sums, the notation~\((\kv)^\oplus\) means the index vector whose \(i\)th~component is~\(\max(k_i,0)\).
In deriving expansion~\eqref{eq:dudxl}, the partial derivatives with respect to~$\xv$ maintain constant \(\Xv\), \(y\) and~\(t\).
Similarly, when taking partial derivatives of any of~\(\Xv\), \(\xv\), \(y\) or~\(t\), the other variables in this group remain constant.

Take the $\ellv$th spatial derivative of the specified \pde~\eqref{eq:upde},
\begin{equation}
\partial_{\xv}^{\ellv}\left(\D tu\right)
=\partial_{\xv}^{\ellv}\left(\sum_{|\kv|=0}^{\infty} \fL_{\kv} \partial_{\xv}^{\kv} u\right)
\implies\D {t}{\big(\partial_{\xv}^{\ellv}u\big)}
=\sum_{|\kv|=0}^{\infty} \fL_{\kv} \partial_{\xv}^{\ellv+\kv} u\,.
\end{equation}
Then substitute~\eqref{eq:dudxl} for the spatial derivatives of field~\(u\) (replacing $\ellv$ with $\ellv+\kv$ where appropriate),
\begin{align}
{}&{}\sum_{|\nv|=0}^{N-|\ellv|-1}\D {t}{u^{(\nv+\ellv)}}\pxvt {\nv}
+\sum_{|\nv|=N}\sum_{\mv=(\nv-\ellv)^\oplus}^{\nv}\binom{\ellv}{\nv-\mv}\partial_{\xv}^{\mv+\ellv-\nv}\D {t}{u^{(\nv)}}\pxvt {\mv}
\nonumber\\{}&{}=
\sum_{|\kv|=0}^{\infty} \fL_{\kv} \sum_{|\nv|=0}^{N-|\ellv+\kv|-1}u^{(\nv+\ellv+\kv)}\pxvt {\nv}
\nonumber\\{}&{}
+\sum_{|\kv|=0}^{\infty} \fL_{\kv} \sum_{|\nv|=N}\sum_{\mv=(\nv-\ellv-\kv)^\oplus}^{\nv}\binom{\ellv+\kv}{\nv-\mv}\partial_{\xv}^{\mv+\ellv+\kv-\nv}u^{(\nv)}\pxvt {\mv}\,,\label{eq:sumpde}
\end{align}
This equation is exact for every \(\xv\in\chi(\Xv)\), for all stations \(\Xv\in\XX\), as the multivariate Taylor multinomial~\eqref{eq:utrtn} is exact.
However, from the last line of~\eqref{eq:sumpde}, regions of rapid spatial variation will have large `uncertain' terms involving~\(\partial_{\xv}^{\kv}{u^{(\nv)}}\) for $|\kv|>0$ and \(|\nv|=N\).

Now set $\xv=\Xv$ in equation~\eqref{eq:sumpde} so that all terms containing factors of $(\xv-\Xv)$ vanish.
Unless otherwise specified, hereafter \(u^{(\nv)}\)~denotes \(u^{(\nv)}(\Xv,y,t)\) when \(|\nv|<N\) and denotes \(u^{(\nv)}(\Xv,\Xv,y,t)\) when \(|\nv|=N\)\,.
In addition, swap the $\nv$ and~$\ellv$ multi-indices in~\eqref{eq:sumpde}.
Then, 
\begin{subequations}\label{eqs:uns}
\begin{equation}
\D {t}{u^{(\nv)}}=\sum_{|\kv|=0}^{N-|\nv|} \fL_{\kv} u^{(\nv+\kv)}+r_{\nv}\,,
\quad\text{for }|\nv|\leq N\,, \label{eq:uodes}
\end{equation}
where the \emph{remainder}  
\begin{equation}
r_{\nv}=\sum_{|\kv|\geq1}\sum_{\substack{|\ellv|=N\\\ellv<\nv+\kv}}\fL_{\kv}\binom{\kv+\nv}{\ellv}\left[\partial_{\xv}^{\kv+\nv-\ellv}u^{(\ellv)}(\Xv,\xv,y,t)\right]_{\xv=\Xv}\,.\label{eq:remain}
\end{equation}
\end{subequations}
The second term on the right-hand side of~\eqref{eq:sumpde} (when $|\kv|\geq1$) determines the remainder~\eqref{eq:remain}.
For all indices \(|\nv|=0,\ldots,N\), the \(u^{(\nv)}\) terms in equation~\eqref{eq:uodes} are evaluated at station~$\Xv$, but the remainder term~$r_{\nv}$ implicitly contains effects due to variations along the line joining  fixed station~\Xv\ to variable position~\xv\ via the integral~\eqref{eq:Nintterm}.

Now rewrite the \textsc{ode}s~\eqref{eq:uodes} of the local field derivatives~$u^{(\nv)}$ in a form corresponding to~\eqref{eq:odes}: 
\begin{equation}
\D {t}{u^{(\nv)}}=\sum_{|\kv|\leq N,\, \kv\geq\nv}\fL_{\kv-\nv} u^{(\kv)}+r_{\nv}\,,\quad\text{for }|\nv|\leq N\,,\label{eq:odesn}
\end{equation}
and with the remainder~$r_{\nv}$ playing the role of time dependent forcing. 
Equation~\eqref{eq:odesn} forms a large system of 
\begin{equation}
\tN:=\binom{N+M}{M}
\label{eq:tN}
\end{equation}
\des{} in~\UU\ since there are ${\tN}$~possible values for $M$~dimensional $\nv$ constrained by $|\nv|\leq N$ (\(\tN=6\) in~\(\RR^2\) for the \(N=2\) example of \autoref{sec:mdhx} which totals \(18\)~\ode{}s).

The system of \tN~\des{}~\eqref{eq:odesn} is an exact statement of the dynamics in the locale~\(\chi(\Xv)\) of every station~\(\Xv\).
The system~\eqref{eq:odesn} might appear closed, but it is coupled to the dynamics of neighbouring stations by the derivatives within the remainder~\eqref{eq:remain}:~\(\partial_{\xv}^{\kv}u^{(\ellv)}\) for \(|\kv|\geq1\) and \(|\ellv|=N\)\,. 
Thus system~\eqref{eq:odesn} is two faced: when viewed globally as the union over all stations~\(\Xv\in\XX\) it is a deterministic autonomous system; but when viewed locally at any one station~\(\Xv\in\XX\)\,, the inter-station coupling implicit in the remainder~\(\rvv_{\nv}\) appears as time dependent `forcing'.

Our plan is to treat the remainders as `uncertainties' and derive models where the effects of the uncertain remainders can be bounded into the precise error statement~\eqref{eq:pderemain} for the models.
Roughly, since the remainder is linear in \(\partial_{\xv}^{\kv}u^{(\ellv)}\propto \partial_{\xv}^{\kv+\ellv}u\) when $|\ellv|=N$~\footnote{The $\kv$th derivative of~\eqref{eq:Nintterm} with respect to $\xv=\Xv+\xiv$ evaluated at $|\xiv|=0$ gives $\partial_{\xv}^{\kv+\ellv}u(\Xv,y,t)=\binom{N+|\kv|}{N}\partial_{\xv}^{\kv}u(\Xv,\Xv,y,t)$\,.}
for slowly varying fields~\(u\) these high derivatives are small and so the errors due to the uncertain remainder will be small.
If the field~\(u\) has any localised internal or boundary layers, then in these locales the errors due to the uncertain remainder will be correspondingly large and so such locales should be excluded from~\XX.

\subsection{The generating function has equivalent dynamics}
\label{sec:gfhed}

From the specified \pde~\eqref{eq:upde} the previous sections dissected out a possibly combinatorially large system of \({\tN}\)~\des{}~\eqref{eq:odesn} describing localised dynamics near any station.
This section uses a multinomial generating function to pack all these \des{} back together into one equation.
The theoretical utility is that we then compactly handle the many \des{} all together.
The practical utility is that the symbology connects with and validates extant, widely~used, heuristic methodologies. 

This section establishes the following proposition.

\begin{proposition}[linear equivalence] \label{thm:lsv}
Let \(u(\xv,y,t)\) be governed by the specified linear \pde~\eqref{eq:upde}.  
Then the dynamics at all locales \(\Xv\in\XX\) are equivalently governed by the \pde
\begin{equation}
\D t\tu=\sum_{|\kv|=0}^{N}\fL_{\kv}\partial_{\xiv}^{\kv}\tu
+\tr[u],
\label{eq:updeng}
\end{equation}
for the generating function multinomial~\(\tu(\Xv,t)\) defined in~\eqref{eq:ugen}, and
for the `uncertain' coupling term~\(\tr[u]\) given by~\eqref{eq:grun}.
\end{proposition}

For every station \(\Xv\in\XX\) and time~\(t\) consider the field~\(u\) in terms of a local Taylor multinomial~\eqref{eq:utrtn} about the cross-section \(\xv=\Xv\)\,.
In terms of the indeterminate \(\xiv\in\RR^M\), define the generating multinomial
\begin{equation}
\tu(\Xv,t):=\sum_{|\nv|=0}^{N-1}\xivt{\nv} u^{(\nv)}(\Xv,y,t)
+\sum_{|\nv|=N}\xivt{\nv}u^{(\nv)}(\Xv,\Xv,y,t),
\label{eq:ugen}
\end{equation}
where this generating multinomial~\(\tilde u\),  through its range denoted by~\(\UU_N\), is implicitly a function of the indeterminate~\xiv, and the cross-sectional variable~\(y\).
This generating multinomial \(\tu:\XX\times\RR\to\UU_N\) for the vector space
\begin{equation}
\UU_N:=\UU\otimes\GG_N
\quad\text{where }
\GG_N:=\{\text{multinomials in \xiv\ of degree}\leq N\},
\label{eq:UUN}
\end{equation}
and where $\otimes$ represents the vector space tensor product. 
In the heat exchanger example of \autoref{sec:sse}, $(u^{00},u^{10},\ldots,u^{02})\in\RR^3\otimes\RR^6=\RR^{18}$ whereas the equivalent generating multinomial is \(\tilde{u}\in\UU_2\) where \(\UU_2= \RR^3\otimes\GG_2\) with $\xiv=(\xi_1,\xi_2)$.
Importantly, the \pde~\eqref{eq:updeng} for~\(\tilde u\) is symbolically the same as the specified \pde~\eqref{eq:upde} with \(u\leftrightarrow\tilde u\), \(\xv\leftrightarrow\xiv\), and the addition of a remainder term~\(\tr[u]\).
But the derivatives~\(\partial_{\xiv}\) in the \pde~\eqref{eq:updeng} are considerably simpler than the derivatives~\(\partial_{\xv}\) in the \pde~\eqref{eq:upde} because in \pde~\eqref{eq:updeng} the derivatives act only on~\(\GG_N\), the space of multinomials in~\xiv\ of degree at most~\(N\).
Simpler because, although derivatives are often confoundingly unbounded in mathematical analysis, here the derivatives~\(\partial_{\xiv}\) are bounded in~\(\GG_N\).

Our first task is to show that the generating multinomial~\eqref{eq:ugen} satisfies the \pde~\eqref{eq:updeng} and to derive an expression for the coupling term~$\tr[u]$.
Our second task is to prove that systematic modelling of the specified \pde~\eqref{eq:upde}, via \pde~\eqref{eq:updeng}, is equivalent to well-known heuristic procedures expressed in terms of the generating multinomial, to some error which is now determined.

To find the remainder~\(\tr[u]\), first observe that \kv~derivatives with respect to the indeterminate~\xiv\ of the generating multinomial~\eqref{eq:ugen} lead to the identity
\begin{equation}
\partial_{\xiv}^{\kv}\tu
=\sum_{|\nv|=0}^{N-|\kv|}\xivt{\nv}u^{(\nv+\kv)}.
\label{eq:dtudxi}
\end{equation}
Second, take the time derivative of~\eqref{eq:ugen} and replace \(\D t{u^{(\nv)}}\) using~\eqref{eq:uodes}:
\begin{align}
\D{t}{\tu}
={}&{}\sum_{|\nv|=0}^{N}\xivt{\nv}
\left[\sum_{|\kv|=0}^{N-|\nv|}\fL_{\kv}u^{(\nv+\kv)}\right]
+\sum_{|\nv|=0}^N\xivt{\nv}r_{\nv}
\nonumber\\
={}&{}\sum_{|\kv|=0}^{N}\fL_{\kv}
\left[\sum_{|\nv|=0}^{N-|\kv|}\xivt{\nv}u^{(\nv+\kv)}\right]
+\sum_{|\nv|=0}^N\xivt{\nv}r_{\nv}
\nonumber\\
={}&{}\sum_{|\kv|=0}^{N}\fL_{\kv}\partial_{\xiv}^{\kv}\tu+\sum_{|\nv|=0}^N\xivt{\nv}r_{\nv}\,.
\label{eq:dttu}
\end{align}
This is precisely \pde~\eqref{eq:updeng} of Proposition~\ref{thm:lsv} with coupling remainder 
\begin{align}
\tr[u]&=\sum_{|\nv|=0}^N\xivt{\nv}r_{\nv}
\nonumber\\
&=\sum_{|\kv|\geq1}\fL_{\kv}\sum_{|\nv|=0}^N\xivt{\nv}\sum_{\substack{|\ellv|=N\\\ellv<\nv+\kv}}
\binom{\kv+\nv}{\ellv}\left[
\partial_{\xv}^{\kv+\nv-\ellv}u^{(\ellv)}(\Xv,\xv,y,t)
\right]_{\xv=\Xv}
\label{eq:grun}
\end{align}
upon using expression~\eqref{eq:remain} for the remainders~$r_{\nv}$.
Expression~\eqref{eq:grun} gives the remainder term appearing in Proposition~\ref{thm:lsv}.

We now turn to the second key task which is to relate fields in physical space with their corresponding field in the generating multinomial space~\(\UU_N\).
Define the operator
\begin{equation}
\cG:=\left[\sum_{|\nv|=0}^N\xivt \nv\partial_{\xv}^{\nv}\right]_{\xv=\Xv}\,,
\label{eq:Gn}
\end{equation}
where these subscripted brackets denote evaluation.
This operator is denoted by~\cG\ to signify it determines the generating multinomial corresponding to a given field. For example, it is straightforward to deduce
\begin{align}&
\cG\left[(\xv-\Xv)^{\nv}\right]=\xiv^{\nv}\quad\text{when }|\nv|\leq N\,, \quad\text{and}
\nonumber\\&
\cG\left[(\xv-\Xv)^{\nv}u^{(\nv)}(\Xv,\xv,y,t)\right] =\xiv^{\nv}u^{(\nv)}(\Xv,\Xv,y,t) \quad\text{when }|\nv|=N\,.
\label{eq:Gop}
\end{align}
Thus $\cG$ operating on the Taylor multinomial expansion~\eqref{eq:utrtn} gives
\begin{equation}
\cG u(\xv,y,t)=\tu(\Xv,t)\in\UU_N=\UU\otimes\GG_N\,.\label{eq:Gopu}
\end{equation}
But to use operator~\(\cG\) on some general function~$f(\xv)\in C^{N+1}$, observe that the Taylor expansion of~\(f(\Xv+\xiv)\) about $\xv=\Xv$  gives
\begin{equation}
\big[f(\xv)\big]_{\xv=\Xv+\xiv}=f(\Xv+\xiv)=\sum_{|\nv|=0}^{N}\xivt \nv\partial_{\xv}^{\nv} f(\Xv)+R_N(f)
=\cG f(\xv)+\Ord{|\xiv|^{N+1}},\label{eq:Gnshft}
\end{equation}
where $R_N(f)$ is the appropriate Lagrange remainder term.
Thus, $\cG f(\xv)$ evaluates~$f(\Xv+\xiv)$ to a difference~$\Ord{|\xiv|^{N+1}}$.

Terms of the form ``\Ord{|\xiv|^{N+1}}'' are not to be viewed as errors, instead they represent differences.
Such differences arise through terms that are of no interest or relevance in this context.
The reason is that we are only interested in operations and identities in the multinomial space~\(\GG_N\) of degree~\(N\) multinomials.
For expressions such as \(f(\Xv+\xiv)\) that typically are off~\(\GG_N\), it is only the projection onto~\(\GG_N\) that is relevant as we only address relations of the components in~\(\GG_N\).
A term ``\Ord{|\xiv|^{N+1}}'' reflects such a projection.

\paragraph{Establish Proposition~\ref{thm:lsv}}
The equivalence of indeterminate~\(\xiv\) and space~\(\xv\), to within the quantified difference, is the key to the equivalence between our rigorous approach to modelling and the well established heuristics of slow scaling of the space variables.
From its definition~\eqref{eq:Gn}, the operator~$\cG$ commutes with $\D t{}$ and thus, from~\eqref{eq:Gopu}, 
\begin{equation}
\cG\D t{u(\xv,y,t)}=\D t{\tu(\Xv,t)}\,.
\end{equation}
However, from~\eqref{eq:Gnshft} we also know that $\cG\D t{u(\xv,y,t)}$ is the Taylor expansion of $\D t{u(\xv,y,t)}$ at $\xv=\Xv+\xiv$ with a difference~$\Ord{|\xiv|^{N+1}}$.
Thus,  to a difference~\Ord{|\xiv|^{N+1}}, the \pde~\eqref{eq:updeng} of the generating multinomial~$\tu(\Xv,t)$ is equivalent  to the \pde~\eqref{eq:upde} of original field~$u(\xv,y,t)$ when evaluated at $\xv=\Xv+\xiv$\,.
At $\xv=\Xv$ (i.e., $\xiv=\ov$) the \pde~\eqref{eq:updeng}  of~$\tu(\Xv,t)$ and the \pde~\eqref{eq:upde} of~$u(\xv,y,t)$ are identical. 
This completes the proof of Proposition~\ref{thm:lsv}.

\subsection{Model the local autonomous dynamics}
\label{sec:mlas}

Having transformed the physical \pde~\eqref{eq:upde} to the equivalent generating function form~\eqref{eq:updeng} we now analyse this later form in order to establish the Slowly Varying \pde\ Proposition~\ref{thm:svpde}.
The generating function form~\eqref{eq:updeng} is symbolically the same as the physical \pde~\eqref{eq:upde} but has two crucial differences: 
firstly, in~\(\GG_N\) the derivatives~\(\partial_{\xiv}^{\nv}\) are bounded and finite-D, whereas in~\XX\ the derivatives~\(\partial_{\xv}^{\nv}\) are unbounded and `infinite'-D;
and secondly, it is the presence of the `uncertain forcing' term~\(\tr[u]\) that couples the local, approximately finite-D, dynamics to the `infinite'-D dynamics over the whole domain~\XX.

To analyse the `uncertainly forced' system~\eqref{eq:updeng} we must first understand the autonomous local system 
\begin{equation}
\D t{\tu}=\tL \tu  \quad\text{with}\quad \tL :=\sum_{|\kv|=0}^{N}\fL_{\kv}\partial_{\xiv}^{\kv} \quad(\tL:\UU_N\to\UU_N).
\label{eq:odesna}
\end{equation}
The invariant subspaces of~\tL\ in~\(\UU_N\) are a key part of our understanding of the autonomous system~\eqref{eq:odesna}.
This subsection establishes Proposition~\ref{thm:ecnddt} that solutions on the centre subspace of~\tL\ satisfy the local, linear, \(N\)th~order, \pde~\eqref{eq:ecnddt}.

Operator~$\tL$ is effectively block upper-triangular. 
The upper triangular nature is due to the derivatives in the definition~\eqref{eq:odesna} of~$\tL$. 
Decompose operator \(\tL=\fL_{\ov}+\tK\), the sum of its `diagonal' part~\(\fL_{\ov}\) and its `off-diagonal' part \(\tK:=\sum_{|\kv|=1}^{N}\fL_{\kv}\partial_{\xiv}^{\kv}\).
Terms of~$\tK\tu$ in~$\xiv^{\nv}$  depend only upon terms of~$\tu$ in~$\xiv^{\kv}$ for \(|\kv|>|\nv|\) ($|\kv|\leq N)$, and so \(\tK\) is zero for all \(|\kv|\leq|\nv|\).
The diagonal block~$\fL_{\ov}$ of the `block upper-triangular' operator~$\tL$ ensures that the spectrum of operator~$\tL$ is that of~$\fL_{\ov}$, but repeated \tN~times (once for each power of~\xiv\ in~\(\GG_N\)), for the combinatorially large~\tN\ defined by~\eqref{eq:tN}.

\begin{assumption} \label{ass:spec}
We assume the following for the primary case of purely centre-stable dynamics.
\begin{enumerate}
\item The Hilbert space~\(\UU\) is the direct sum of two closed \(\fL_{\ov}\)-invariant subspaces, \(\EE_c^0\) and~\(\EE_s^0\), and the corresponding restrictions of~\(\fL_{\ov}\) generate strongly continuous semigroups \cite[e.g.]{Gallay93, Aulbach96}.
\item The operator~\(\fL_{\ov}\) has a discrete spectrum of eigenvalues~\(\lambda_1,\lambda_2,\ldots\) (repeated according to multiplicity) with corresponding linearly independent (generalised) eigenvectors~\(v_1^{\ov},v_2^{\ov},\ldots\) that are complete (\(\UU=\Span\{v_1^{\ov},v_2^{\ov},\ldots\}\)). 
\item The first \(m\)~eigenvalues \(\lambda_1,\ldots,\lambda_m\) of~\(\fL_{\ov}\) all have real part satisfying \(|\Re\lambda_j|\leq\alpha\) and hence the \(m\)-dimensional \emph{centre subspace} \(\EE_c^0=\Span\{v_1^{\ov},\ldots,v_m^{\ov}\}\) \cite[Chap.~4, e.g.]{Chicone2006}.
\item All other eigenvalues~\(\lambda_{m+1},\lambda_{m+2},\ldots\) have real part negative and well separated from the centre eigenvalues, namely \(\Re\lambda_j\leq-\beta<-N\alpha\) for \(j=m+1,m+2,\ldots\)\,, and so the stable subspace \(\EE_s^0=\Span\{v_{m+1}^{\ov},v_{m+2}^{\ov},\ldots\}\).
\end{enumerate}
\end{assumption}

Although we almost entirely address the case when the Hilbert space~$\UU$ decomposes into only a centre and a stable subspace, much of the following derivation and discussion applies to other cases that may be of interest in other circumstances.  
One may be interested in a centre subspace among both stable and unstable modes, or in a \emph{slow subspace} corresponding to pure zero eigenvalues (as in the random walker\slash  heat exchanger example of \autoref{sec:mdhx}), or in some other `normal mode' \emph{subcentre subspace} \cite[e.g.]{Lamarque2011}, or in the \emph{centre-unstable subspace}, and so on.
We primarily focus on the centre subspace among otherwise decaying modes as then the centre subspace contains the long term emergent dynamics from quite general initial conditions (\cite{Robinson96} called it asymptotically complete).

\begin{definition} \label{def:baseig}
Recall Assumption~\ref{ass:spec} identifies a subset of~$m$ eigenvectors of~$\fL_\ov$ which span the centre subspace~$\EE^0_c\ (\subset\UU)$.
\begin{itemize}
\item With these eigenvectors define 
\begin{equation*}
V^\ov
:=\begin{bmatrix} v_1^\ov&v_2^\ov&\cdots&v_m^\ov \end{bmatrix}
\in\UU^{1\times m}.
\end{equation*}

\item 
Since the centre subspace is an invariant space of~$\fL_\ov$, define \(A_{\ov}\in\RR^{m\times m}\) to be such that \(\fL_\ov V^\ov =V^\ov A_{\ov}\) (often \(A_{\ov}\) will be in Jordan form, but it is not necessarily so).

\item
Use \(\langle\cdot,\cdot\rangle\) to also denote the inner product on the Hilbert space~$\UU$, $\langle\cdot,\cdot\rangle:\UU\times\UU\to\RR$\,.

Interpret this inner product when acting on two matrices\slash vectors with elements in  \(\UU\) as the matrix\slash vector of the corresponding elementwise inner products.
For example, for $Z^{\ov}, V^{\ov}\in \UU^{1\times m}$\,, $\langle Z^{\ov},V^{\ov}\rangle\in \RR^{m\times m}$.

\item
Also define \(Z^{\ov}\) to have \(m\)~columns, linearly independent, which both span the corresponding centre subspace of the adjoint~\(\Tr{\fL_\ov}\) (in the chosen inner product), and also are normalised such that \(\langle Z^\ov,V^\ov\rangle=I_m\)\,. 

\end{itemize}
\end{definition}

Developing from this eigen-decomposition of~\(\cL_\ov\), we require \(m\tN\)~basis vectors of~$\tL$ to span the centre subspace of~\tL, denoted~$\EE^N_c$, in the generating multinomial space~$\UU_N$. 
We now establish the existence of {suitable} basis vectors of~$\tL$ from those of~$\fL_\ov$ by considering suitable multinomials in~$\xiv$.
These basis vectors are typically generalised eigenvectors of~\tL\ in the multinomial space~\(\UU_N\)---they are derived from, but are very different to, the eigenvectors of~\(\fL_\ov\)

\paragraph{Recursively define generalised eigenvectors}
After solving the basic eigenproblem for~\( A_{\ov}\), \(V^{\ov}\) and~\(Z^{\ov}\), Definition~\ref{def:baseig}, now recursively solve the following sequence of problems for \( A_{\nv}\in\RR^{m\times m}\)~and~\(V^{\nv}\in\UU^{1\times m}\), \(0<|{\nv}|\leq N\)\,, 
\begin{subequations}\label{eqs:toep}%
\begin{eqnarray}&&
 A_{\nv}:=
 \sum_{0<|\kv|, \kv\leq \nv}\big\langle Z^{\ov},\fL_{\kv} V^{\nv-\kv}\big\rangle,
\label{eq:toeplam}
\\&&
\fL_{\ov}V^{\nv}-V^{\nv} A_{\ov}=
-\sum_{0<|\kv|, \kv\leq \nv}\fL_{\kv}V^{\nv-\kv}
+ \sum_{0<|\kv|, \kv\leq \nv}V^{\nv-\kv} A_{\kv}\,,
\label{eq:topevn}
\\&&\big\langle Z^{\ov},V^{\nv}\big\rangle=0_m\,.
\label{eq:toeport}
\end{eqnarray}%
\end{subequations}
In applications, the \(m\)~columns of each of these~\(V^{\nv}\) contain information about the interactions between plate-wise gradients of the field~\(u\), as felt through the mechanisms encoded in the \(\fL_{\kv}\) for $|\kv|>0$ and the cross-sectional out-of-equilibrium dynamics encoded in~\(\fL_{\ov}\).

\begin{lemma}
  The recursive equations~\eqref{eqs:toep} are solvable for~\(A_{\nv}\) and \(V^\nv\) for all \(0<|{\nv}|\leq N\)\,.
\end{lemma}
\begin{proof} 
Consider \(\langle Z^\ov,\eqref{eq:topevn}\rangle\).
The left-hand side, using the choice~\eqref{eq:toeplam} and the orthogonality~\eqref{eq:toeport}, becomes
\begin{eqnarray*}&&
\big\langle Z^\ov,\fL_{\ov}V^\nv\big\rangle-\big\langle Z^\ov,V^\nv A_{\ov}\big\rangle
=\big\langle \Tr\fL_{\ov}Z^\ov,V^\nv\big\rangle-\big\langle Z^\ov,V^\nv\big\rangle A_{\ov}
\\&&{}
=\big\langle Z^\ov\Tr A_{\ov},V^\nv\big\rangle-0_m A_{\ov}
= A_{\ov}\big\langle Z^\ov,V^\nv\big\rangle
= A_{\ov}0_m
=0_m\,.
\end{eqnarray*}
Whereas the right-hand side of~\(\langle Z^\ov,\eqref{eq:topevn}\rangle\), also using the orthogonality~\eqref{eq:toeport}, becomes
\begin{eqnarray*}&&
-\sum_{0<|\kv|, \kv\leq \nv}\big\langle Z^\ov,\fL_\kv V^{\nv-\kv}\big\rangle
+\sum_{0<|\kv|<|\nv|, \kv\leq \nv}\big\langle Z^\ov,V^{\nv-\kv} A_{\kv}\big\rangle
+\big\langle Z^\ov,V^\ov A_{\nv}\big\rangle
\\&&{}
=-\sum_{0<|\kv|, \kv\leq \nv}\big\langle Z^\ov,\fL_\kv V^{\nv-\kv}\big\rangle
+\sum_{0<|\kv|<|\nv|, \kv\leq \nv}\big\langle Z^\ov,V^{\nv-\kv}\big\rangle A_{\kv}
+\big\langle Z^\ov,V^\ov\big\rangle A_{\nv}
\\&&{}
=-\sum_{0<|\kv|, \kv\leq \nv}\big\langle Z^\ov,\fL_\kv V^{\nv-\kv}\big\rangle
+\sum_{0<|\kv|<|\nv|, \kv\leq \nv}0_m A_{\kv}
+I_m A_{\nv}
\\&&{}
=-\sum_{0<|\kv|, \kv\leq \nv}\big\langle Z^\ov,\fL_\kv V^{\nv-\kv}\big\rangle
+ A_{\nv}
=0_m
\end{eqnarray*}
by the choice~\eqref{eq:toeplam}.
Consequently the right-hand side of~\eqref{eq:topevn} is in the range of the left-hand side.
Also, the left-hand side operator has a null space spanned by the \(m\)~columns of~\(V^\ov\) and so there is enough freedom to impose the normality condition~\eqref{eq:toeport} at every step. 
\end{proof}

\begin{lemma}\label{lem:curlyV}
For the homogeneous system~\eqref{eq:odesna} in the multinomial space~\(\UU_N\), a basis for the centre subspace of~\tL\ is the collective columns of 
\begin{equation}
\tV^\nv:=\sum_{\kv=\ov}^\nv V^{\nv-\kv}\xivt{\kv}
\,,
\quad 0\leq|\nv|\leq N\,.
\label{eq:genevec}
\end{equation}
Let \tV\ denote the collection of columns of~\(\tV^\nv\) over all~\nv, \(0\leq|\nv|\leq N\)\,, ordered within the partial ordering of~\(|\nv|\).
Denote the centre subspace of~\tL, spanned by columns of~\tV, by~\(\EE_c^N\).
\end{lemma}
\begin{proof} 
First prove the space spanned by~\eqref{eq:genevec} is invariant: for all~\nv, \(0\leq|\nv|\leq N\)\,, consider
\begin{eqnarray}
\tL\tV^\nv
&=&\sum_{|\ellv|=0}^N\fL_\ellv\partial_\xiv^\ellv \tV^\nv
=\sum_{|\ellv|=0}^N\fL_\ellv\partial_\xiv^\ellv \sum_{\kv=\ov}^\nv V^{\nv-\kv}\xivt{\kv}
\nonumber\\
&=& \sum_{\kv=\ov}^\nv \sum_{\ellv=\ov}^\kv \fL_\ellv  V^{\nv-\kv}\xivt{(\kv-\ellv)}
= \sum_{\kv=\ov}^\nv \sum_{\ellv=\ov}^\kv \fL_\ellv  V^{\nv-\kv}\xivt{(\kv-\ellv)}
\nonumber\\
&=& \sum_{\kv=\ov}^\nv \sum_{\ellv=\ov}^\kv \fL_{\kv-\ellv}  V^{\nv-\kv}\xivt{\ellv}
= \sum_{\ellv=\ov}^\nv \xivt{\ellv}
\sum_{\kv=\ellv}^\nv \fL_{\kv-\ellv}  V^{\nv-\kv}
\nonumber\\
&=& \sum_{\ellv=\ov}^\nv \xivt{\ellv} \sum_{\kv=\ov}^{\nv-\ellv} \fL_{\kv}  V^{\nv-\ellv-\kv}
= \sum_{\ellv=\ov}^\nv \xivt{\ellv} \sum_{\kv=\ov}^{\nv-\ellv} V^{\nv-\ellv-\kv}A_{\kv}
\nonumber\\
&= &\sum_{\kv=\ov}^{\nv} \Bigg[\sum_{\ellv=\ov}^{\nv-\ellv} \xivt{\ellv}  V^{\nv-\kv-\ellv}\Bigg]A_{\kv}
= \sum_{\kv=\ov}^{\nv} \tV^{\nv-\kv}A_{\kv} \,,
\label{eq:tlvva}
\end{eqnarray}
which is a linear combination of~\(\{\tV^\nv\}\), the columns of~\tV.
Since this identity holds for all~\nv, there exists an \(m\tN\times m\tN\) matrix~\bA, composed of~\(\tN\times\tN\)  blocks of \((m\times m)\) \(A_\nv\) and~\(0_m\), such that \(\tL\tV=\tV\bA\)\,. 
(The random walker\slash heat exchanger matrix~\bA\ in~\eqref{eq:A} furnishes an example of such a block upper-triangular matrix~\bA---in a case with \(m=1\)\,.)

Second, from the identity~\eqref{eq:tlvva}, \(\tL\tV^\nv=\tV^\nv A_\ov+(\text{lower order }\tV^\ellv)\), where `lower order' means \(|\ellv|<|\nv|\).   
Consequently, matrix~\bA\ is upper-triangular.
Further, the \tN~diagonal blocks of~\bA\ are~\(A_\ov\). 
Thus the eigenvalues corresponding to the eigenspace spanned by~\tV\ are the centre eigenvalues in~\(A_\ov\) repeated \(\tN\)~times.
These fully account for the centre eigenvalues of~\tL\ (counted according to multiplicity).

Third, the linear independence of both~\(\{\xivt\nv\}\)\ in~\(\GG_N\) and the columns of~\(V^\ov\) in~\UU\ imply, via definition~\eqref{eq:genevec}, that the columns of~\tV\ are linearly independent in \(\UU_N=\UU\otimes\GG_N\)\,.
\end{proof}

Our aim is to show the evolution on the centre subspace has a physically appealing, compact, representation directly corresponding to the physical space \pde~\eqref{eq:dcdta}.
Further, the representation sets up connections to other established methodologies. 
The next proposition is this desired result.

\begin{proposition} \label{thm:ecnddt}
Let \(\{U_\nv\in\RR^m : \nv\in\NN_0^M,\ 0\leq|\nv|\leq N\}\) parametrise the centre subspace~\(\EE_c^N\) via \(\tu=\sum_{|\nv|=0}^N\tV^\nv U_\nv\)\,.
For these parameters, define the generating multinomial \(\tU(\xiv,t):=\sum_{|\nv|=0}^N \xivt{\nv}U_\nv(t)\). 
Then on the centre subspace~\(\EE_c^N\) of~\tL\ the evolution of the autonomous~\eqref{eq:odesna} is governed by the \pde
\begin{equation}
\D t\tU=\sum_{|\nv|=0}^N A_{\nv}\partial_\xiv^\nv\tU\,,
\label{eq:ecnddt}
\end{equation}
for matrices~\(A_\nv\) constructed by~\eqref{eq:toeplam}.
\end{proposition}

\begin{proof} 
Consider the autonomous~\eqref{eq:odesna}, \(\D t\tu=\tL\tu\): by the parametrisation its left-hand side \(\D t\tu=\sum_{|\nv|=0}^N\tV^\nv \D t{U_\nv}\)\,; whereas the right-hand side
\begin{eqnarray*}
\tL\tu&=&\tL\sum_{|\nv|=0}^N\tV^\nv U_\nv =\sum_{|\kv|=0}^N\tL\tV^\kv U_\kv
=\sum_{|\kv|=0}^N\sum_{\nv=\ov}^\kv \tV^{\kv-\nv} A_{\nv} U_\kv
\quad(\text{by \eqref{eq:tlvva}})
\\&=&\sum_{|\kv|=0}^N\sum_{\nv=\ov}^\kv \tV^{\nv} A_{\kv-\nv} U_\kv
=\sum_{|\nv|=0}^N\tV^{\nv} \sum_{\kv\geq\nv,|\kv|\leq N}  A_{\kv-\nv} U_\kv
\\&=&\sum_{|\nv|=0}^N\tV^\nv\sum_{|\ellv|=0}^{N-|\nv|} A_{\ellv} U_{\ellv+\nv}\,.
\end{eqnarray*}
By the linear independence of the columns of~\tV, these two sides are equal iff
\begin{equation}
\D t{U_\nv}=\sum_{|\ellv|=0}^{N-|\nv|} A_{\ellv}U_{\ellv+\nv}
\quad\text{for all }|\nv|\leq N\,.
\label{eq:utsumau}
\end{equation}
Second, consider the time derivative of the generating multinomial~\tU:
\begin{eqnarray*}
\D t\tU&=&\D t{}\sum_{|\nv|=0}^N \xivt{\nv}U_\nv
=\sum_{|\nv|=0}^N \xivt{\nv}\D t{U_\nv}
\\&=&\sum_{|\nv|=0}^N \xivt{\nv}\sum_{|\ellv|=0}^{N-|\nv|} A_{\ellv}U_{\ellv+\nv}
=\sum_{|\nv|=0}^N \sum_{|\ellv|=0}^{N-|\nv|}  A_{\ellv}\xivt{\nv}U_{\ellv+\nv}
\\&=&
\sum_{|\nv|=0}^N \sum_{|\ellv|=0}^{N-|\nv|}  A_{\ellv}\partial_\xiv^\ellv\xivt{(\nv+\ellv)}U_{\ellv+\nv}
=\sum_{|\ellv|=0}^N A_{\ellv}\partial_\xiv^\ellv\sum_{|\nv|=0}^{N-|\ellv|}  \xivt{(\nv+\ellv)}U_{\ellv+\nv}
\\&&
=\sum_{|\ellv|=0}^N A_{\ellv}\partial_\xiv^\ellv\sum_{|\kv|=0}^{N}  \xivt{\kv}U_{\kv}
=\sum_{|\ellv|=0}^N A_{\ellv}\partial_\xiv^\ellv\tU
\,.
\end{eqnarray*}
Thus the \pde~\eqref{eq:ecnddt} holds on the autonomous centre subspace of~\tL\ in~\(\UU_N\).
\end{proof}

This result establishes that locally, and in the absence of coupling with nearby locales, a macroscale field together with its slowly-varying derivatives exist that evolve consistently with the expected macroscale \pde.

As an example we return to the random walker of \autoref{sec:mdhx} and show how to construct all $\tV^{\nv}$ for $|\nv|\leq\tN=6$.
We use initial basis vectors $V^{\ov}=v_1^{00}=(1,0,0)$ (the eigenvector of $\fL_{\ov}$ with zero eigenvalue) and \(Z^{\ov}=(1,0,0)\) (the  eigenvector of adjoint $\fL_{\ov}^{\dag}$ with zero eigenvalue).
Either, we first recursively calculate all $A_{\nv}$ and basis vectors $V^{\nv}$ from~\eqref{eqs:toep} and use definition~\eqref{eq:genevec} to construct all  $\tV^{\nv}$; or, we calculate $A_{\nv}$ and $\tV^{\nv}$ directly with \begin{subequations}\label{eqs:toepG}%
\begin{eqnarray}&&
 A_{\nv}:=
 \sum_{0<|\kv|, \kv\leq \nv}\big\langle Z^{\ov},\fL_{\kv} \tV^{\nv-\kv}\big\rangle_{\xiv=\ov}\,,
\\&&
\fL_{\ov}\tV^{\nv}-\tV^{\nv} A_{\ov}=
-\sum_{0<|\kv|, \kv\leq \nv}\fL_{\kv}\tV^{\nv-\kv}
+ \sum_{0<|\kv|, \kv\leq \nv}\tV^{\nv-\kv} A_{\kv}\,,
\\&&\big\langle Z^{\ov},\tV^{\nv}\big\rangle=\xivt{\nv}\,,
\end{eqnarray}%
\end{subequations}
obtained from combining \eqref{eqs:toep}~and~\eqref{eq:genevec}.
In either case, $A_{00},A_{01},A_{11}=0$\,, $A_{10}=-\tfrac13$\,, $A_{20}=\tfrac{8}{27}$ and $A_{02}=\tfrac23$ (in agreement with matrix $\bA$ in~\eqref{eq:A}) and
\begin{align*}
\tV^{00}&=V^{00}=(1,0,0)\,,\\
\tV^{10}&=V^{10}+V^{00}\xi_1=(\xi_1,0,-\tfrac29)\,,\\
\tV^{01}&=V^{01}+V^{00}\xi_2=(\xi_2,-1,0)\,,\\
\tV^{20}&=V^{20}+V^{10}\xi_1+V^{00}\tfrac12\xi_1^2=(\tfrac12\xi_1^2,0,-\tfrac{4}{81}-\tfrac29\xi_1)\,,\\
\tV^{11}&=V^{11}+V^{10}\xi_1+V^{01}\xi_2+V^{00}\xi_1\xi_2=(\xi_1\xi_2,\tfrac89-\xi_1,-\tfrac29\xi_2)\,,\\
\tV^{02}&=V^{02}+V^{01}\xi_2+V^{00}\tfrac12\xi_2^2=(\tfrac12\xi_2^2,-\xi_2,\tfrac19)\,.
\end{align*}
These three dimensional generating basis vectors of order $N=2$ multinomials are equivalent to the six $3\tN=18$~dimensional eigenvectors $\vv^{mn}$ in equation~\eqref{eqs:bss6}.
Straightforward substitution confirms that $A_{\nv}$ and basis vectors~$\tV^{\nv}$  satisfy identity~\eqref{eq:tlvva}.

\subsection{Project the uncertain coupling}
\label{sec:hruf}

Our aim is not to model the autonomous~\eqref{eq:odesna}, but the exact system~\eqref{eq:updeng} with its uncertain coupling~\(\tr[u]\) to neighbouring locales.
Let's proceed to project the uncertain coupling by treating it as an \emph{arbitrary}, time dependent, forcing.
The result of this subsection completes the proof of Proposition~\ref{thm:svpde}.

\paragraph{Change basis to centre and stable variables}
Write \(\tu=\tV \cU+\tW \cS\) where the centre variables~\(\cU=\begin{bmatrix} U_\nv \end{bmatrix}\) parametrise the centre subspace, and the stable variables~\(\cS\) parametrise the stable subspace.
Here, $\tW$ is a differential matrix operator containing eigenvectors that form a basis for the multinomial stable subspace~$\EE^N_s$ and is analogous to~$\tV$, the differential operator containing eigenvectors forming a basis fo the multinomial centre subspace~$\EE^N_c$ (Lemma~\ref{lem:curlyV}).
As detailed in Assumption~\ref{ass:spec}, the eigenvectors of the stable subspace have eigenvalues~$\lambda_j$ where $\Re\lambda_j\leq-\beta<-N\alpha$\,, indicating rapid exponential decay of these modes and the emergence of the centre subspace over long time scales (with its eigenvalues $|\Re\lambda_j|\leq \alpha$). 
Despite the rapid decay of the stable modes, when forced by coupling with neighbouring locales, their influence on the dynamics of the system is not negligible in general.
Here we derive a slow macroscale model which quantifies the effects of the coupling.

Analogous to~$\tV$,  $\tW$ is associated with the following properties:
\begin{itemize}
\item \tW\ forms a basis for the multinomial stable subspace~\(\EE_s^N\) of~\tL;
and
\item there exists an operator~\(\cB:\EE_s^N\to\EE_s^N\) such that \(\tL\tW=\tW\cB\) and all eigenvalues of~\cB\ have real part\({}\leq-\beta\) (the eigenvalues of~$\cB$ are $\lambda_{m+1},\lambda_{m+2},\ldots$).

\end{itemize}
We separate the `forcing' in system~\eqref{eq:updeng} into components in~\(\EE_c^N\) and~\(\EE_s^N\): \(\tr=\tV\tr_c+\tW\tr_s\)\,.
Then by the linear independence of the complete basis~\{\tV,\tW\}, the `forced' system~\eqref{eq:updeng} separates to
\begin{subequations}\label{eqs:ecde}%
\begin{eqnarray}&&
\D t{\cU}=\bA \cU+\tr_c\,,
\label{eq:ecdec}
\\&&
\D t{\cS}=\cB \cS+\tr_s
\,.
\label{eq:ecdeq}
\end{eqnarray}%
\end{subequations}

Now consider the stable variables~$\cS$.
Since \tL\ generates a continuous semigroup (Assumption~\ref{ass:spec}),
so does its restriction~\cB, and so we rewrite~\eqref{eq:ecdeq} in the integral equation form
\begin{equation}
\cS(t)
=e^{\cB t}\cS(0)+\int_0^t e^{\cB (t-\tau)}\tr_s(\tau)\,d\tau
=e^{\cB t}\cS(0)+e^{\cB t}\star \tr_s,
\label{eq:qfntime}
\end{equation}
as convolutions \(f(t)\star g(t)=\int_0^tf(t-\tau)g(\tau)\,d\tau\)\,.  Since all eigenvalues of~\cB\ have real part\({}\leq-\beta\), then for some decay rate~\(\gamma\in(\alpha,\beta)\)
\begin{equation}
\cS(t)=e^{\cB t}\star \tr_s+\Ord{e^{-\gamma t}},
\quad\text{written}\quad
\cS(t)\simeq e^{\cB t}\star \tr_s
\label{eq:rforcedq}
\end{equation}
upon invoking the following definition. 
\begin{definition} \label{def:simeq}
Define \(f(t)\simeq g(t)\) to mean \(f-g=\Ord{e^{-\gamma t}}\) as \(t\to\infty\) for some exponential decay rate~\(\gamma\) such that \(\alpha<\gamma <\beta\)\,.
\end{definition}  
Consequently, equation~\eqref{eq:rforcedq} quantifies how the local stable variables~\(\cS\) are forced by  coupling with neighbouring locales via the remainder effects in~\(\tr_s\).

\paragraph{The centre subspace dynamics with remainder}
As invoked in Proposition~\ref{thm:svpde}, define the macroscale amplitude field of slowly varying solutions by the projection
\begin{equation}
U(\xv,t):=\langle Z^\ov,u(\xv,y,t)\rangle,
\label{eq:cssf}
\end{equation}
which as yet is distinct from the multinomial local centre variables~\(\cU\).
In order to discover how the amplitude field~\(U(\xv,t)\) evolves, our task is to now relate the field~\(U(\xv,t)\) to the local centre subspace variables~\(\cU\).

As a preliminary step, and since~$Z^\ov$ is independent of~$\xiv$, for any index~\ellv\ simplify
\begin{eqnarray}
\big\langle Z^\ov,\partial_\xiv^\ellv\tV\big\rangle_{\xiv=\ov}
&=&\left\langle Z^\ov,\begin{bmatrix} \partial_\xiv^\ellv\tV^\nv \end{bmatrix}_{\xiv=\ov}\right\rangle
\quad(\text{by Lemma~\ref{lem:curlyV}, with matrix index~\nv})
\nonumber\\&=&\left\langle Z^\ov,\begin{bmatrix} V^{\nv-\ellv}\text{ for }\nv\geq\ellv,\text{ else }0 \end{bmatrix}\right\rangle
\quad(\text{by \eqref{eq:genevec}})
\nonumber\\&=&\begin{bmatrix} \langle Z^\ov,V^{\nv-\ellv}\rangle\text{ for }\nv\geq\ellv,\text{ else }0_m \end{bmatrix}
\quad(\text{by Defn~\ref{def:baseig}})
\nonumber\\&=&\begin{bmatrix} I_m\text{ for }\nv=\ellv,\text{ else }0_m \end{bmatrix}
\quad(\text{by Defn~\ref{def:baseig} and~\eqref{eq:toeport}}).
\label{eq:zdlv}
\end{eqnarray}
Consequently, with \(\ellv=\ov\), \(\langle Z^\ov,\tV\rangle_{\xiv=\ov}\bA\) is the first row of blocks of~\bA, namely the \(m\times m\) blocks~\(A_\nv\) in appropriate order.
Recalling that \(\cU=\begin{bmatrix} U_\nv \end{bmatrix}\), an identity to be used shortly is then that
\begin{equation}
\langle Z^\ov,\tV\rangle_{\xiv=\ov}\bA\cU
=\sum_{|\nv|=0}^NA_\nv U_\nv\,.
\label{eq:zvau}
\end{equation}

Recall that the multinomial~$\tu(\Xv,t)$ is the Taylor expansion in~\xiv\ of~$u(\Xv+\xiv,y,t)$  to differences~$\Ord{|\xiv|^{N+1}}$.
Hence the field $u(\Xv,y,t)=\tu(\Xv,t)\big|_{\xiv=\ov}$ and so macroscale amplitude $U(\Xv,t)=\langle Z^\ov,\tu(\Xv,t)\big|_{\xiv=\ov}\rangle$.
Take the temporal derivative of~$U(\xv,t)$ at cross-section $\xv=\Xv$\,, 
\begin{align}
\D t{U(\Xv,t)}={}&{}\left\langle Z^\ov,\D t{\tu(\Xv,t)}\right\rangle_{\xiv=\ov}
\nonumber\\
={}&{}\left\langle Z^\ov,\tV\D t{\cU}\right\rangle_{\xiv=\ov}+\left\langle Z^\ov,\tW\D t{\cS}\right\rangle_{\xiv=\ov}
\nonumber\\
&(\text{using separated system \eqref{eqs:ecde}})
\nonumber\\
={}&{}\langle Z^\ov,\tV\bA\cU\rangle_{\xiv=\ov}+\langle Z^\ov,\tV \tr_c\rangle_{\xiv=\ov}
+\langle Z^\ov,\tW\cB\cS\rangle_{\xiv=\ov}+\langle Z^\ov,\tW \tr_s \rangle_{\xiv=\ov}
\nonumber\\
&(\text{using convolution solution \eqref{eq:rforcedq}})
\nonumber\\
={}&{}\langle Z^\ov,\tV\rangle_{\xiv=\ov}\bA\cU +\langle Z^\ov,\tV \tr_c\rangle_{\xiv=\ov}
+\langle Z^\ov,\tW\cB e^{\cB t}\star \tr_s\rangle_{\xiv=\ov}
\nonumber\\{}&{}
+\langle Z^\ov,\tW \tr_s \rangle_{\xiv=\ov}+\Ord{e^{-\gamma t}}
\nonumber\\
&(\text{using identity~\eqref{eq:zvau} and Defn~\ref{def:simeq}})
\nonumber\\
\simeq{}&{}\sum_{|\nv|=0}^NA_\nv U_\nv\ +\langle Z^\ov,\tV \tr_c\rangle_{\xiv=\ov}
+\langle Z^\ov,\tW\cB e^{\cB t}\star \tr_s\rangle_{\xiv=\ov}
\nonumber\\{}&{}
+\langle Z^\ov,\tW \tr_s \rangle_{\xiv=\ov} .
\label{eq:dcdt}
\end{align}

For this result to form a \pde\ for the macroscale field~\(U\) we need to write the centre subspace parameters~\(U_\nv\) in terms of spatial derivatives of~\(U\).
Identity~\eqref{eq:dtudxi} ensures $\partial_{\xiv}^{\kv}\tu\big|_{\xiv=0}=u^{(\kv)}=\partial_{\xv}^{\kv}u\big|_{\xv=\Xv}$ for $|\kv|\leq N$\,. 
Then the $\ellv$th~spatial derivative of $U(\xv,t)$ at $\xv=\Xv$ is 
\begin{align}
\partial_{\xv}^{\ellv}U\big|_{\xv=\Xv}
={}&{}\partial_{\xv}^{\ellv}\langle Z^\ov,u\rangle\big|_{\xv=\Xv}
=\langle Z^\ov,\partial_{\xv}^{\ellv} u|_{\xv=\Xv}\rangle=
\langle Z^\ov,\partial_{\xiv}^{\ellv} \tu \rangle_{\xiv=\ov}
\nonumber\\&\quad(\text{using }\tu=\tV\cU+\tW\cS)
\nonumber\\={}&{}
\langle Z^\ov,\partial_{\xiv}^{\ellv} \tV\rangle_{\xiv=\ov}\cU +\langle Z^\ov,\partial_{\xiv}^{\ellv} \tW\cS \rangle_{\xiv=\ov}
\nonumber\\&\quad(\text{using \eqref{eq:zdlv} and the solution~\eqref{eq:rforcedq}})
\nonumber\\={}&{}
\begin{bmatrix} I_m\text{ for }\nv=\ellv,\text{ else }0_m \end{bmatrix}\cU
+\langle Z^\ov,\partial_{\xiv}^{\ellv} \tW e^{\cB t}\star \tr_s \rangle_{\xiv=\ov}+\Ord{e^{-\gamma t}}
\nonumber\\\simeq{}&{}
U_\ellv+\langle Z^\ov,\partial_{\xiv}^{\ellv} \tW e^{\cB t}\star \tr_s \rangle_{\xiv=\ov} .
\label{eq:dc}
\end{align}
The above shows that, discounting exponential transients, $U_\ellv$~is the $\ellv$th~spatial derivative of the amplitude field~$U(\xv,t)$ evaluated at $\xv=\Xv$\,, with a remainder term determined from the forcing coupling.

Substituting equation~\eqref{eq:dc} into~\eqref{eq:dcdt} gives
\begin{equation}
\left.\D t{U}\right|_{\xv=\Xv}\simeq\sum_{|\nv|=0}^{N}
\left.A_{\nv}\partial_{\xv}^{\nv}U\right|_{\xv=\Xv}+\rho\,,
\label{eq:cpde}
\end{equation}
where the remainder 
\begin{align}
\rho={}&{}\langle Z^\ov,\tV \tr_c\rangle_{\xiv=\ov}+\langle Z^\ov,\tW\cB e^{\cB t}\star \tr_s\rangle_{\xiv=\ov}+\langle Z^\ov,\tW \tr_s \rangle_{\xiv=\ov}
\nonumber\\{}&{}
-\sum_{|\nv|=0}^{N}A_{\nv}\langle Z^\ov,\partial_{\xiv}^{\nv} \tW e^{\cB t}\star \tr_s \rangle_{\xiv=\ov}\,.\label{eq:pderemain}
\end{align}

The \pde~\eqref{eq:cpde} applies for every station~\(\Xv\) in the domain~\(\XX\).
Strictly, the `\pde'~\eqref{eq:cpde} is a coupled differential-integral equation: the dynamics at each station~\(\Xv\) being coupled by the gradients and their history convolution integrals occurring within the remainder~\eqref{eq:pderemain}.
But when the uncertain remainder term is negligible, as in slowly varying regimes where the remainder~\(\rho\) is~\Ord{\sum_{|\nv|=N+1}|\partial_\xv^\nv u|}, then equation~\eqref{eq:cpde} reduces to the plate \pde\ closure~\eqref{eq:dcdta}.  
This completes the argument that establishes Proposition~\ref{thm:svpde}.

\section{Application: homogenisation of multiscale diffusion in 2D}
\label{sec:ehmd2d}
\def\ur{{\mathrm u}} \def\xr{{\mathrm x}} \Vec{xr}

Many engineering structures have microscale structure, such as the windings in electrical machinery \cite[e.g.]{Romanazzi2016}, 
electromagnetism in micro-structured material \cite[e.g.]{Craster2015, Niyonzima2014}, and
slow Stokes flow through porous media \cite[e.g.]{Brown2011}.
The engineering challenge is to understand the dynamics on a scale significantly larger than the micro-structure.
Homogenization, via multiple length scales, is the common approach \cite[e.g.]{Gustafsson03, Engquist08}.
Building on the 1D case \cite[\S2.5]{Roberts2013a}, this section shows a new approach to modelling the large scale dynamics within general rigorous theory.
This new approach provides a route to systematic refinements of the homogenization, and to novel quantification of the remainder error.

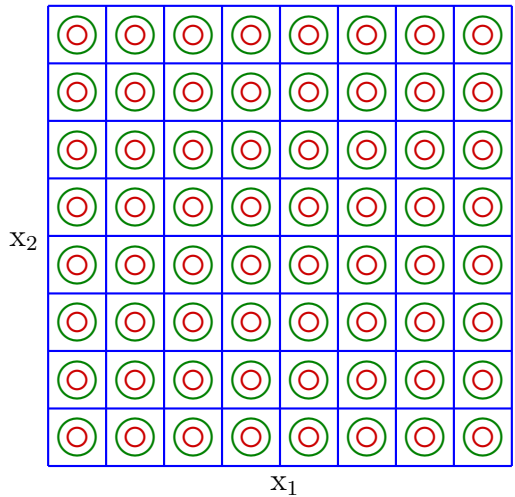
\begin{figure}
\centering
\setlength{\unitlength}{0.7ex}
\begin{picture}(48,48)
\put(23,-2.5){$\xr_1$}
\put(-4,23){$\xr_2$}
\thicklines
\color{blue}
\multiput(0,0)(6,0)9{\line(0,1){48}}
\multiput(0,0)(0,6)9{\line(1,0){48}}
\color{green!50!black}
\multiput(3,3)(6,0)8{\multiput(0,0)(0,6)8{\circle4}}
\color{red!80!black}
\multiput(3,3)(6,0)8{\multiput(0,0)(0,6)8{\circle2}}
\end{picture}
\caption{microscale periodic cells of size \(h\times h\) in 2D are represented by a spatially varying, doubly \(h\)-periodic, diffusion coefficient~\(K(\xrv)\) in the \pde~\eqref{eq:svdpde0} over some large domain.}
\label{fig:cells}
\end{figure}

In two spatial dimensions the prototypical problem is the effective diffusion across structured material with a periodic, cellular, microscale (\autoref{fig:cells}). 
Let \(\xrv=(\xr_1,\xr_2)\) be the spatial coordinates, then we seek to model on the macroscale---that is, across many cells---the diffusion in time~\(t\) of the temperature field~\(\ur(\xrv,t)\) satisfying the diffusion \pde\
\begin{equation}
\D t\ur=\D{\xr_1}{}\left[K(\xrv)\D{\xr_1}\ur\right]
+\D{\xr_2}{}\left[K(\xrv)\D{\xr_2}\ur\right].
\label{eq:svdpde0}
\end{equation}
The spatially varying diffusion coefficient~\(K(\xrv)\) is here assumed to be doubly periodic over a length~\(h\) as illustrated schematically by \autoref{fig:cells}; that is, \(K(\xr_1+ph,\xr_2+qh)=K(\xr_1,\xr_2)\) for integer~\(p,q\).
Here we use upright roman letters for field~\(\ur\) and space~\xrv\ to distinguish these direct physical quantities from those of the mathematical analysis which use the maths font~\(u\) and~\xv\ for closely related but different quantities.

Ensemble averages provides our route to rigorous modelling.
Let's embed the specific problem in the family of problems of all phase shifted versions of the material.
That is, for all microscale phase shifts~\phiv, \(0\leq\phi_1,\phi_2<h\)\,, seek the field \(\ur(\xrv,\phiv,t)\) that satisfies the \pde
\begin{equation}
\D t\ur=\D{\xr_1}{}\left[K(\xrv+\phiv)\D{\xr_1}\ur\right]
+\D{\xr_2}{}\left[K(\xrv+\phiv)\D{\xr_2}\ur\right].
\label{eq:svdpde}
\end{equation}
The original \pde~\eqref{eq:svdpde0} and its solution is included in this family as the phase shift \(\phiv=\ov\) version of \pde~\eqref{eq:svdpde} and its solution.
The second step in the embedding is to write the solution field~\(\ur\) in terms of a new field~\(u(\xv,\yv,t)\) that is a function of macroscale coordinates \(\xv\in\XX\), microscale `cell' coordinates \(\yv\in\YY=[0,h]^2\), and time~\(t\).
Hereafter let's consider the diffusivity~\(K\) to be only a function of the microscale cell coordinate~\yv, that is,~\(K(\yv)\).
Then consider solutions~\(u(\xv,\yv,t)\) to the \pde
\begin{eqnarray}
\partial_tu&=&\partial_{y_1}\big[K(\yv)\partial_{y_1}\big]u
+\partial_{y_2}\big[K(\yv)\partial_{y_2}\big]u
\nonumber\\&&{}
+\big[K_{y_1}+2K\partial_{y_1}\big]u_{x_1}
+\big[K_{y_2}+2K\partial_{y_2}\big]u_{x_2}
\nonumber\\&&{}
+K(\yv)u_{x_1x_1}+K(\yv)u_{x_2x_2}\,.
\label{eq:emdpde}
\end{eqnarray}
Elementary algebra shows that solutions of this \pde~\eqref{eq:emdpde}, via \(\ur(\xr,\phiv,t)=u(\xr,\xr+\phiv,t)\), give solutions to the \pde~\eqref{eq:svdpde}---and hence~\eqref{eq:svdpde0}---and vice-versa.

\begin{figure}
\centering
\includegraphics{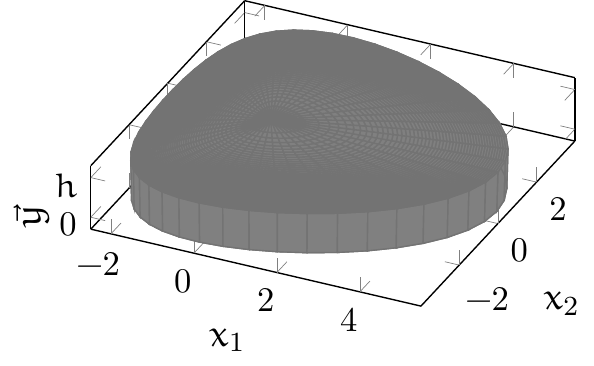}\hfil 
\includegraphics{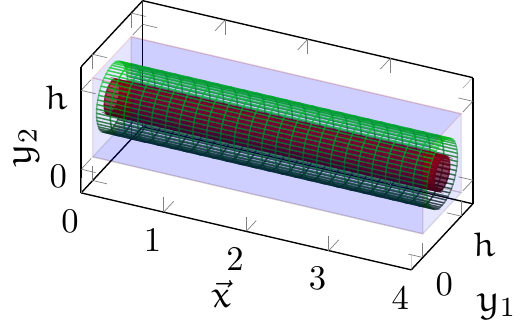}
%
%
%
%
\caption{schematic diagrams both illustrating that at every macroscale position~\(\xv\in\XX\), the \pde~\eqref{eq:emdpde} has a small square cross-section \(\YY=[0,h]^2\) of one cell.  Solutions of \pde~\eqref{eq:svdpde} are the values of field~\(u\) on any 2D slice through these pictures of \(\yv=(\xv+\phiv)\) mod\(~h\).}
\label{fig:cellxsec}
\end{figure}

We need to consider boundary conditions for both domains~\XX\ and~\YY.
First, the slowly varying theory of \autoref{sec:pmid} applies usefully in a macroscale domain~\XX\ from which boundary layers and shocks have been excised, and addresses the general solutions in such a domain~\XX.
Thus macroscale boundary conditions on~\(\partial\XX\) are not needed to apply the theory.
Second, on the microscale~\YY, the diffusion coefficient is \(h\)-periodic in~\yv, so we correspondingly require the field~\(u(\xv,\yv,t)\) to be \(h\)-periodic in~\yv.
Then identities such as \(\yv=\xrv+\phiv\) are implicitly to be interpreted modulo~\(h\) in both components.
This periodicity in the cell structure is sufficient to form a well-posed problem in the cross-sectional, cell, microstructure.

The embedding \pde~\eqref{eq:emdpde} is in the linear class~\eqref{eq:upde} of \autoref{sec:pmid} with \(M=2\) large space dimensions and operators
\begin{eqnarray*}
\fL_\ov&=&\partial_{y_1}\big[K(\yv)\partial_{y_1}\big]
+\partial_{y_2}\big[K(\yv)\partial_{y_2}\big],
\\\fL_{(1,0)}&=&\big[K_{y_1}(\yv)+2K(\yv)\partial_{y_1}\big]
\\\fL_{(0,1)}&=&\big[K_{y_2}(\yv)+2K(\yv)\partial_{y_2}\big]
\\\fL_{(2,0)}&=&\fL_{(0,2)}=K(\yv),
\end{eqnarray*}
and all other~\(\fL_\kv\) zero.
Other crucial requisites are those of Assumption~\ref{ass:spec} upon the properties of~\(\fL_\ov\) with boundary conditions of double \(h\)-periodicity in \(\yv\in\YY\); that is, upon the basic properties of the fundamental cell problem.
The self-adjoint cell eigen-problem
\begin{equation*}
\fL_\ov v=\partial_{y_1}\big[K(\yv)v_{y_1}\big]
+\partial_{y_2}\big[K(\yv)v_{y_2}\big]=\lambda v
\end{equation*}
is well known and for diffusion \(K(\yv)\geq K_{\min}>0\) has \(m=1\) zero eigenvalue corresponding to eigenfunction \(v_0(\yv)={}\)constant over the cell~\YY, and the other eigenvalues are real and negative, \(\lambda_j\leq -\beta\) for \(\beta=4\pi^2K_{\min}/h^2\).
The corresponding set of eigenfunctions are orthogonal and complete in the Hilbert space of smooth~\(L^2\) functions on~\YY. 
Consequently the requisite semigroups exist \cite[Ch.~6, e.g.]{Carr81}.
Hence Proposition~\ref{thm:svpde} applies.
Set \(Z^\ov(\yv)=1/h^2\) so that the macroscale field \(U(\xv,t):=\frac1{h^2}\iint_\YY u(\xv,\yv,t)\,d\yv\) is a cell-mean.
Then, for any chosen order of truncation~\(N\), the field~\(U\) satisfies the \pde
\begin{equation}
\D tU=\sum_{|\nv|=1}^NA_\nv\Dn\xv\nv U\,,\quad \xv\in\XX\,,
\label{eq:smdpde}
\end{equation}
for some \(2\times2\) matrices~\(A_\nv\) depending upon~\(K(\yv)\), to a closure remainder error quantified by~\eqref{eq:pderemain}, and upon neglecting transients decaying roughly like~\Ord{e^{-\beta t}} as \(t\to\infty\)\,.

The \pde~\eqref{eq:smdpde} is the homogenisation of the original \pde~\eqref{eq:svdpde0}, generalised to any order of truncation.
Further, the novel expression~\eqref{eq:pderemain} would be used to quantify the remainder error in any large scale modelling.
\footnote{One subtlety in the result is that in any given physical realisation, the ensemble of initial conditions has to be chosen so that the remainder error becomes small.  
If an ensemble of initial conditions were chosen poorly, then the remainder error would stay large over space-time.}

\paragraph{Vibrations of an inhomogeneous plate}
If, instead of the inhomogeneous diffusion~\eqref{eq:svdpde0}, suppose we wanted to model the vibrations of an inhomogeneous plate satisfying the corresponding, second order in time, \pde
\begin{equation}
\DD t\ur=\D{\xr_1}{}\left[K(\xrv)\D{\xr_1}\ur\right]
+\D{\xr_2}{}\left[K(\xrv)\D{\xr_2}\ur\right].
\label{eq:svdpdev0}
\end{equation}
Then all of the above algebra would be effectively the same except for two aspects.
First, the exponential emergence of the model from all initial conditions, expressed in terms~``\Ord{e^{-\gamma t}}'', would be replaced by long-lived oscillations~\(e^{i\omega_j t}\) of relatively high frequency.
Second, also the exponentially decaying convolutions in the remainder term~\(r\) would turn into tricky convolutions with oscillating factors~\(e^{i\omega_j t}\), potentially causing an algebraic growth in the error of the model.
Both of these mechanisms would be ameliorated by any small viscous damping or small radiative damping that is physically present but omitted from the mathematical \pde~\eqref{eq:svdpdev0}.
Consequently, in practice and with care the approach here should help model the vibration of inhomogeneous plates.

\section{Conclusion}

This article develops a general theoretical approach to supporting the much invoked practical approximation of slow variations in space.
The key idea, suggested by \cite{Roberts2013a}, is to examine the dynamics in the locale around any cross-section.
We find that a Taylor series approximation to the dynamics is only coupled to neighbouring locales via the highest order resolved derivative. 
Treating this coupling as an `uncertain forcing' of the local dynamics we in essence apply non-autonomous centre manifold theory to prove the existence and emergence of a local model.
This support applies for all cross-sections and so establishes existence and emergence globally in the domain.

In this theory there is no requirement for some small parameter to tend to zero. 
A centre manifold model exists for solutions up to at least some finite amplitude and up to at least some finite spatial gradients of the variables.
Thus the approach remains valid when the domain of the original system is finite in all dimensions.

This article focussed on the case of a centre manifold amongst centre-stable dynamics as this case is the most broadly useful in modelling dynamics.
\autoref{sec:mdhx} considered one such example with slow-stable dynamics and derived the slowly varying model on the slow manifold.  
The key required properties are the persistence of centre manifolds under perturbations by time dependent `forcing'.
Since this property of persistence is shared by other invariant manifolds, we expect the same approach will support the existence and perhaps relevance of other invariant manifolds with slow variations in space.
The persistence of centre manifold should also support the extension of the theory to stochastic systems~\cite[e.g.]{Arnold98, Roberts06k}, and to systems with nonlinear dynamics, as shown by \cite{Roberts2013a} for the one dimensional case $\XX\subset \RR$. 

This approach opens much for future research.
It may be able to illuminate the thorny issue of providing boundary conditions to slowly varying models of problems such as shells, plates and Turing patterns \cite[e.g]{Segel69, Roberts92c, Mielke91a}.

\paragraph{Acknowledgement} The Australian Research Council Discovery Project grant DP120104260 and DP150102385 helped support this research.
We thank Arthur Norman and colleagues who maintain the Reduce software.

{\raggedright\printbibliography}

\appendix

\section{Notation}
\label{sec:nota}
This appendix summarises a lot of the symbols used in the general theory.  The third column in the table of notation gives the case corresponding to the random walker\slash heat exchanger of \autoref{sec:mdhx}.

\begin{longtable}{@{}lp{0.55\linewidth}p{0.21\linewidth}@{}}
\hline
Symbol&Meaning&Random walker
\\\hline\endhead
\(\xv\in\XX\)& `large' open spatial domain in~\(\RR^M\) of the plate
&\((x,y)\in\RR^2\)
\\\(y\in\YY\)& space of the `cross-section' of the plate
&\(\YY=\{0,1,2\}\)
\\\(\UU\)&Hilbert space of field values&\(\RR^3\)
\\\(u(\xv,y,t)\)& function of field values, in~\(C^{2N}\)
&\((u_0,u_1,u_2)\)
\\\(u^{(\nv)}(\Xv,y,t)\)
&for \(|\nv|<N\)\,, the \nv{}th derivative in~\xv\ evaluated at \(\xv=\Xv\)
&
\\\(u^{(\nv)}(\Xv,\xv,y,t)\)
&for \(|\nv|=N\)\,, weighted integral, between~\Xv\ and~\xv, of the \nv{}th derivative in~\xv
&
\\\(N\)&some chosen order of truncation of the multivariable Taylor series
&\(N=2\)
\\\(\fL_\kv\)& maps \(\UU\to\UU\), operator coefficients of the \kv{}th~derivative in~\xv\ in the original \pde
&\(3\times3\) matrices
\\\(\NN_0\)
&\(\{0,1,2,\ldots\}\), the non-negative integers
\\\kv, \nv, \ldots&\(M\)-dimensional integer multi-indices in \(\NN_0^M\), \(|\kv|\leq N\)
&\(\kv=(k_1,k_2),\ldots\)
\\\(U(\xv,t)\)&macroscale emergent variables, in~\(\RR^m\);
in general \(U:=\langle Z^\ov,u(\xv,y,t)\rangle\)
&\(U_0(x,y,t)\)
\\\(A_\nv\)&\(m\times m\) matrix coefficient of the \nv{}th spatial derivative of~\(U\) in the macroscale \pde
& scalars \(0\), \(-\frac13\), \(\frac8{27}\), \(\frac23\)
\\\(\sum_{|\kv|=a}^b\)
&denotes a sum over all indices \(\kv\in\NN_0^M\) satisfying \(a\leq|\kv|\leq b\).
\\\(\sum_{|\kv|=a}^\infty\)
&denotes a sum over all indices \(\kv\in\NN_0^M\) satisfying \(|\kv|\geq a\)\,, but the sum truncates as there are a finite number of non-zero~\(\fL_\kv\).
\\\(\sum_{\text{condition}}\)
&denotes a sum over all variable indices that satisfy the specified condition.
\\\(\sum_{\kv=\av}^{\bv}\)
&denotes a sum over indices~{\kv} such that \(\av\leq\kv\leq\bv\)\,, that is, \(a_\ell\leq k_\ell\leq b_\ell\) for all \(\ell=1,2,\ldots,M\)\,.
\\\(r,r_\nv\)&complicated remainder terms for various expressions, functionals of the field~\(u\)
\\\tN& \({}=\binom{N+M}M\) is the number of \des, each in~\UU, of the dynamics for any locale in~\XX
&\(\tN=6\)
\\\(\xiv\in\RR^M\)&generating multinomial variable---is effectively a local space variable, \(\xv=\Xv+\xiv\)
\\\(\GG_N\)& space of multinomials in~\xiv\ of degree\({}\leq N\)
\\\(\UU_N\)& \({}=\UU\otimes\GG_N\) the space of multinomial fields
\\\tu(\Xv,t)& the generating multinomial with coefficients~\(u^{(\nv)}\), depends implicitly on~\xiv\ and~\(y\)
\\\(\tr\)&generating multinomial of remainder terms, a complicated functional of field~\(u\)
\\\cG&operator to give the generating multinomial, at~\Xv, of any given field
\\\tL& \(\sum_{|\kv|=0}^{N}\fL_{\kv}\partial_{\xiv}^{\kv}\) differential operator on multinomials, corresponds to given \pde
\\\(\EE_c^0,\EE_s^0\)
&subspaces of~\UU, invariant under~\(\fL_\ov\), centre and stable respectvely
\\\(\lambda_j,v_j^\ov\)
&complete eigenvalues and (generalised) eigenvectors (in~\UU) of~\(\fL_\ov\)
&eigenvalues \(0,-1,-3\)
\\\(V^\ov\in\UU^{1\times m}\)
&\({}=\begin{bmatrix} v_1^\ov&v_2^\ov&\cdots&v_m^\ov \end{bmatrix}\), \(m\) columns are a basis for centre subspace~\(\EE_c^0\)
&\((1,0,0)\)
\\\(Z^\ov\in\UU^{1\times m}\)
&\(m\)~columns are a basis for the centre subspace of the adjoint~\(\Tr{\fL_\ov}\), also \(\langle Z^\ov,V^\ov\rangle=I_m\)
&\((1,0,0)\)
\\\(\EE_c^N\subset\UU_N\)
&multinomial centre subspace of~\tL
\\\(V^\nv\in\UU^{1\times m}\)
&each are components of centre eigenvectors of~\tL, derived recursively 
\\\(\tV^\nv\in\UU_N^{1\times m}\)
&\(m\)~columns are linearly independent centre eigenvectors of~\tL,  
\\\(\tV\in\UU_N^{1\times m\tN}\)
&\({}=[\tV^\nv]\), its \(m\tN\) columns are a basis for~\(\EE_c^N\)
\\\bA
&matrix of \(\tN\times\tN\) blocks of \(A_\nv\) and~\(0_m\), as appropriate.
&the \(6\times6\) matrix in~\eqref{eq:A}
\\\(U_\nv\in\RR^m\)
&each are \(m\)~parameters of the multinomial centre subspace~\(\EE_c^N\), \(m\tN\)~in total 
&\(U_{(m,n)}=U_0^{mn}\)
\\\(\tU\in\GG_N^m\)
&corresponding generating multinomial of the centre subspace parameters
\\\(\cU\in\RR^{m\tN}\)
&\({}=\begin{bmatrix} U_\nv \end{bmatrix}\), all parameters of the multinomial centre subspace of~\tL
&\(\cU=\big[U_0^{mn}\big]\)
\\\(\EE_s^N\subset \UU_N\)
&multinomial stable subspace of~\tL
\\\(\tW\)
&a basis for the multinomial stable subspace~\(\EE_s^N\) of~\tL\ (formal)
\\\(\cS\)
&parameters of the multinomial stable subspace of~\tL
&\(\cS=\big[U_j^{mn}\big]\), \(j=1,2\)
\\\(\cB\)
&restriction of multinomial operator~\tL\ to \(\EE_s^N\) 
\\\(\tr_c,\tr_s\)
&components of the remainder term~\tr\ in the subspaces \(\EE_c^N,\EE_s^N\) respectively
\\\hline
\end{longtable}

\clearpage

\section{Computer algebra models the random walker}
\label{sec:camhe}

\input{camhe}

\end{document}

%% file: camhe.tex

This section lists and describes computer algebra code to analyse the Taylor series approach to the slowly varying modelling of the heat exchanger~\eqref{eqs:hedim} of Figure~\ref{fig:heat2d}.
The code is written in the free computer algebra package Reduce\footnote{\url{http://www.reduce-algebra.com/} gives full information about Reduce.} \cite[e.g.]{MacCallum91}.
Analogous code will work for other computer algebra packages.

Make the printing appears nicer.
\begin{reduce}
on div; on revpri; off allfac; 
\end{reduce}

\subsection{Transform PDEs to ODEs for Taylor coefficients}

Define the $x$ and $y$ components of the velocities in each plane.
The matrices $\texttt{vx}$ and $\texttt{vy}$ are diagonal with $\vec{v}_j=(\texttt{vx}_{jj},\texttt{vy}_{jj})$\,.
Also define the mixing operator \texttt{lop} which describes the rate of the walker changing among the three directions.
\begin{reduce}
nn:=2;
vx:=mat((1,0,0),(0,-1,0),(0,0,+1));
vy:=mat((1,0,0),(0, 0,0),(0,0,-1));
lop:=mat((-1,1,0),(1,-2,1),(0,1,-1))$ 
\end{reduce}

Define the model in terms of the heat flows $c_j(x,y,t)$.
\begin{reduce}
array resc(3);
operator ct;
depend ct,t,x,y;
for j:=1:3 do resc(j):=-df(ct(j),t)+(for k:=1:3 sum (
    lop(j,k)*ct(k)-vx(j,k)*df(ct(k),x)-vy(j,k)*df(ct(k),y) ))$
\end{reduce}
Map from the $c$~fields to the slow and fast $u$~heat fields~u\ defined by~\eqref{eq:udefs}
The $u_0(x,y,t)$~field is slow with eigenvalue $\lambda_0=0$\,, whereas fields $u_1(x,y,t)$ and~$u_2(x,y,t)$ are fast with eigenvalues $\lambda_1=-1$ and $\lambda_2=-3$\,, respectively. 
\begin{reduce}
array resu(2);
operator ut;  
depend ut,t,x,y;
ct(1):=(ut(0)+ut(1)+ut(2))$
ct(2):=(ut(0)-2*ut(2))$
ct(3):=(ut(0)-ut(1)+ut(2))$
write resu(0):=(resc(1)+resc(2)+resc(3))/3;
write resu(1):=(resc(1)-resc(3))/2;
write resu(2):=(resc(1)-2*resc(2)+resc(3))/6;
array evl(2);
evl(0):=0$ evl(1):=-1$ evl(2):=-3$
\end{reduce}

Extract coefficient matrices of these modal \ode{}s for later use.
\begin{reduce}
matrix ll00(3,3),ll10(3,3),ll01(3,3);
for i:=0:2 do ll00(i+1,i+1):=evl(i);
for i:=0:2 do for j:=0:2 do ll10(i+1,j+1):=df(resu(i),df(ut(j),x));
for i:=0:2 do for j:=0:2 do ll01(i+1,j+1):=df(resu(i),df(ut(j),y));
\end{reduce}

Construct the Taylor expansion of the $u$~fields to order~$N$ using Lagrange's remainder theorem with Taylor coefficients~$u_{jmn}$ where $j=0,1,2$, and $m,n=0,1,\ldots,N$ with $|(m,n)|\leq N$\,.
Each coefficient~$u_{jmn}$ is a function of~$X$, $Y$~and~$t$, but those with $|(m,n)|=N$ are also functions of $x$~and~$y$.
\begin{reduce}
operator u; depend u,xx,yy,t; 
for j:=0:2 do for k:=0:nn do depend u(j,k,nn-k), x,y$
for j:=0:2 do ut(j):=(for m:=0:nn sum for n:=0:nn-m sum 
    u(j,m,n)*(x-xx)^m/factorial(m)*(y-yy)^n/factorial(n))$
\end{reduce}

Obtain the \textsc{ode} for each Taylor coefficient.
\begin{reduce}
array odeu(2,nn,nn);
for j:=0:2 do for m:=0:nn do for n:=0:nn-m do 
    write odeu(j,m,n):=sub(y=yy,x=xx,df(df(resu(j),x,m),y,n))$
\end{reduce}

The uncertain `forcing' terms are $\partial_x u_{j}^{(m,n)}$ and $\partial_y u_{j}^{(m,n)}$ with order $m+n=N$\,.
Rename these $w(j,m,n,z)$ where the subscript refers to the derivative~$\partial_z$ with $z=x,y$\,. 
(Because \verb|sub()| is still active, we have to replace xx by x and yy by y.)
\begin{reduce}
inclw:=0; 
operator w; depend w,tt;
for j:=0:2 do for m:=0:nn do for n:=0:nn-m do  
  write odeu(j,m,n):=((odeu(j,m,n) 
    where df(u(~k,~l,~p),~z)=>w(k,l,p,z)*inclw when z neq t)
    where {xx=>x, yy=>y});
depend tt,t;
\end{reduce}

\subsection{Initialise the construction of a transform}

We want to determine a new set of fields~$U_{jmn}$ for which the evolution~$\dot{U}_{jmn}$ separates fast and slow fields. 
Store the transform for \(u_j^{(m,n)}\) in array \verb|ux(j,m,n)|, and the  right-hand side of the corresponding \ode\ for \(U_j^{(m,n)}\) in array \verb|duudt(j,m,n)|
\begin{reduce}
operator uu; depend uu,xx,yy,t;
let { df(uu(~p,~q,~r),t)=>duudt(p,q,r)
       , u(~p,~q,~r)=>ux(p,q,r) };
array ux(2,nn,nn), duudt(2,nn,nn);
\end{reduce}

As a first approximation the transform $u\mapsto U$ is the identity and $\dot{U}_{jmn}=\lambda_jU_{jmn}$\,.
\begin{reduce}
for m:=0:nn do for n:=0:nn-m do for j:=0:2 do begin 
    ux(j,m,n):=uu(j,m,n);
    duudt(j,m,n):=evl(j)*uu(j,m,n);
end;
\end{reduce}

Need to express the uncertain remainders as history integrals so use well established operators \cite[e.g.]{Roberts06k, Roberts06j}.
\begin{reduce}
operator z; linear z;
let { df(z(~f,tt,~mu),t)=>-sign(mu)*f+mu*z(f,tt,mu)
    , z(1,tt,~mu)=>1/abs(mu)
    , z(z(~r,tt,~nu),tt,~mu) =>
      (z(r,tt,mu)+z(r,tt,nu))/abs(mu-nu) when (mu*nu<0)
    , z(z(~r,tt,~nu),tt,~mu) =>
      -sign(mu)*(z(r,tt,mu)-z(r,tt,nu))/(mu-nu)
      when (mu*nu>0)and(mu neq nu)
    };
\end{reduce}

Define an operator to separate out terms in fast stable variables.
\begin{reduce}
operator fast; linear fast;
let { fast(uu(0,~m,~n),t)=>0 
    , fast(uu(~j,~m,~n),t)=>uu(j,m,n) when j>0 
    , fast(w(~a,~b,~c,~d),t)=>0
    , fast(z(~a,tt,~b),t)=>0 };
\end{reduce}

The above properties are \emph{critical}: they must be correct for the results to be correct.

Determine the effect on the slow manifold of the fast forcing.
\begin{reduce}
operator slow; linear slow;
let { slow(uu(~j,~m,~n),uu)=>uu(j,m,n)/evl(j) };
\end{reduce}

\subsection{Iterate to separate slow-fast coordinates}

Iterate to obtain the transform $u\mapsto U$ and the evolution $\dot{U}_{jmn}$\,.
The evolution of the slow fields $U_{0mn}$ should only contain slow fields, whereas the evolution of the fast fields, $U_{1mn}$~and~$U_{2mn}$, should only contain fast fields.
For $j=1,2$\,, all fast fields in the residue of the \textsc{ode} of fast field $u_{jmn}$ are placed in the evolution $\dot{U}_{jmn}$ and all remaining terms are placed in the transform of $u_{jmn}$\,.
All fast fields in the residue of the \textsc{ode} of slow field $u_{0mn}$ are placed in the transform of $u_{0mn}$ and all remaining terms are placed in the evolution~$\dot{U}_{0mn}$.
The coupling terms~$w$ may have both slow and fast components but are designated slow.
\begin{reduce}
for iter:=1:9 do begin
  ok:=1; lengthRess:={};
  for m:=0:nn do for n:=0:nn-m do begin
    res:=odeu(0,m,n); 
    lengthRess:=length(res).lengthRess;
    ux(0,m,n):=ux(0,m,n)+slow(gd:=fast(res,t),uu);
    duudt(0,m,n):=duudt(0,m,n)+(res-gd);
    if res neq 0 then ok:=0;

    for j:=1:2 do begin
    res:=odeu(j,m,n); 
    lengthRess:=length(res).lengthRess;
    duudt(j,m,n):=duudt(j,m,n)+(gd:=fast(res,t));
    ux(j,m,n):=ux(j,m,n)+z(res-gd,tt,evl(j));
    if res neq 0 then ok:=0;
    end;
  end;
  write lengthRess:=lengthRess;
  showtime;
  if ok then write iter:=iter+10000;
end;
\end{reduce}
Check the iteration converged to the specified order.
\begin{reduce}
if not ok then rederr("The iteration failed to converge");
\end{reduce}

\subsection{Write out the transform}
On completing the iteration, write all transforms and evolutions.
\begin{reduce}
for j:=0:2 do for n:=0:nn do for m:=0:n do 
    write ux(j,n-m,m):=ux(j,n-m,m);
for j:=0:2 do for n:=0:nn do for m:=0:n do 
    write duudt(j,n-m,m):=duudt(j,n-m,m);
\end{reduce}

\subsection{Check the generating multinomial form}

First find the slow subspace eigenvectors in the multinomial form: they come from the various coefficients of each of the \verb|uu(0,m,n)| amplitudes.
\begin{reduce}
array vt(nn,nn);
for n:=0:nn do for m:=0:n do begin
  vt(n-m,m):=tp mat((0,0,0));
  for p:=0:nn do for q:=0:nn-p do vt(n-m,m):=vt(n-m,m)+df(
    tp mat((ux(0,p,q),ux(1,p,q),ux(2,p,q)))
    ,uu(0,n-m,m))*xi^p*yi^q/factorial(p)/factorial(q);
end;
\end{reduce}
Find that \(\D{\xi_1}{\vv^{(m,n)}}=\vv^{(m-1,n)}\) and \(\D{\xi_2}{\vv^{(m,n)}}=\vv^{(m,n-1)}\).  This pattern must be useful.

Then put these eigen-multinomials into one whole
(works for any basis vectors at all).
The second version below uses \verb|dx| and \verb|dy| to notionally symbolise differential operators in some manner yet to be decided.
\begin{reduce}
p:=-1$ factor zz;
vtt:=for n:=0:nn sum for m:=0:n sum vt(n-m,m)*zz^(p:=p+1);
vtt:=for n:=0:nn sum for m:=0:n sum vt(n-m,m)*dx^(n-m)*dy^m;
\end{reduce}

Now pre-multiply by~\tL\ obtained from the basic modal \ode{}s:
\begin{reduce}
lltvtt:=ll00*vtt+ll10*df(vtt,xi)+ll01*df(vtt,yi);
\end{reduce}

Is this the same as the following?
\begin{reduce}
matrix avtt(3,1);
for i:=1:3 do avtt(i,1):=
  (duudt(0,0,0) where uu(0,~m,~n)=>df(vtt(i,1),xi,m,yi,n));
errorInOps:=avtt-lltvtt;
\end{reduce}
Yes it is!
So, post-multiplying by the eigen-matrix is equivalent to premultiplying by the slow evolution operator.
And this happens automatically.

Evaluating \verb|avtt| or \verb|lttvtt| at \(\xi=\zeta=0\) then gives  in its first row the required differential operator of the slow evolution, and something as yet undecided in the 2nd and 3rd row---must be something to do with the generalised eigenvectors of the slow modes.

\subsection{Write transform in LaTeX}
Write out in pretty LaTeX.
The definition of the \LaTeX\ command is a bit dodgy as convolutions of convolutions are not printed in the correct order; however,  convolutions commute so it does not matter.
\begin{reduce}
load_package rlfi;
mathstyle math;
defindex ux(down,up,up);
defindex duudt(down,up,up);
defindex uu(down,up,up);
defindex w(down,up,up,down);
defid uu,name="U";
defid ux,name="u";
defid duudt,name="\dot U";
defid tt,name="}{";
defid w,name="u";
defid z,name="\z";
\end{reduce}

Change \verb|name| to get braces, not left-right parentheses.
\begin{reduce}
deflist('((!( !{) (!) !}) ),'name)$
\end{reduce}
Force all fractions (coded in Reduce as \verb|quotient|) to use \verb|\frac| command so we can change how it appears.
\begin{reduce}
put('quotient,'laprifn,'prinfrac);
\end{reduce}

Write expressions to the file \verb|glmsmvs.red| for later reading.
Prepend the expressions with an instruction to write a heading, and surround the heading with anti-math mode to cancel the math environment that rlfi puts in.
\begin{reduce}
out "glmsmvs.red"$
for j:=0:2 do for m:=0:nn do for n:=0:m do 
    write "ux(",j,",",m-n,",",n,"):=ux(",j,",",m-n,",",n,");";
for j:=0:2 do for m:=0:nn do for n:=0:m do 
    write "duudt(",j,",",m-n,",",n,"):=duudt(",j,",",m-n,",",n,");";
write "end;";
shut "glmsmvs.red";
\end{reduce}

Now write the LaTeX:
\begin{reduce}
out "glmsmvs.ltx"$
on latex;
write "
\newcommand{\z}[2]{e^{\ifnum#2=-1 -\else#2\fi t}{\star}#1}
";
in "glmsmvs.red"$
off latex;
shut "glmsmvs.ltx"$
\end{reduce}

\begin{reduce}
end;end;end;
\end{reduce}

\subsection{Transform back to slowly varying}

We now treat the slow fields~$U_{0mn}$ as the Taylor coefficients of some slow field function $U(x,y,t)$ and construct a slow field \textsc{pde}.
\begin{reduce}
operator uufun; depend uufun,x,y,t;
\end{reduce}

All field coefficients~$U_{0mn}$ are functions of $X$, $Y$ and~$t$, but those field coefficients with $|(m,n)|=N$ are also functions of $x$~and~$y$.
\begin{reduce}
for j:=0:2 do for k:=0:nn do depend uu(j,k,nn-k),x,y$
\end{reduce}

Define the Taylor series expansion of $U(x,y,t)$ with coefficients $U^c_{0mn}=U_{0mn}$\,.
\begin{reduce}
operator uut; depend uut,x,y,t;
operator uuc; depend uuc,t;
for j:=0:2 do uut(j):=(for m:=0:nn sum for n:=0:(nn-m) sum 
    uuc(j,m,n)*(x-xx)^m/factorial(m)*(y-yy)^n/factorial(n));
\end{reduce}

Construct a \textsc{pde} for $U(x,y,t)$ which, like the original $u(x,y,t)$~\textsc{pde}, is first order in time.
For clarity, we place $w_x(u_{jmn})$~and~$w_y(u_{jmn})$ for $j=1,2$ and all convolutions in the function `force'.
\begin{reduce}
resuu:=(-df(uufun(0),t)+(df(uut(0),t)) 
    where {uuc(~k,~m,~n)=>uu(k,m,n)})$
forced:={z0(~f,t)=>0, z1(~f,t)=>0, z2(~f,t)=>0,
    wx(1,~k,nn-k)=>0, wy(1,~k,nn-k)=>0, wx(2,~k,nn-k)=>0,
    wy(2,~k,nn-k)=>0}$
force:=resuu-(resuu where forced)$
resuu:=resuu-force$
for m:=0:nn do resuu:=sub(
  {wy(0,m,nn-m)=df(ux(0,m,nn-m),y)
  ,wx(0,m,nn-m)=df(ux(0,m,nn-m),x)}
  ,resuu);    
\end{reduce}

Replace all coefficient fields~$U_{0mn}$ with derivatives of~$U(x,y,t)$ and its Taylor series:
\begin{equation}
U_{0mn}=\partial_x^{m}\partial_y^nU(x,y,t)-[\partial_x^{m}\partial_y^nU_{\text{Taylor}}(x,y,t)-U^c_{0mn}]\,.
\end{equation}
\begin{reduce}
for l:=0:nn do for k:=0:l do begin
  resuu:=(resuu where {uu(0,l-k,k)=>
  df(uufun(0),x,l-k,y,k)-(df(uut(0),x,l-k,y,k)-uuc(0,l-k,k))})$
  resuu:=(resuu where {uuc(0,~m,~n)=>uu(0,m,n)});
end;
\end{reduce}

In the \textsc{pde} of~$U(x,y,t)$ the forcing terms are defined by the function `force'.
\begin{reduce}
factor df;  
resuu:=resuu+forcing;  
\end{reduce}

End the if-statement that chooses whether to execute the code of this appendix.
\begin{reduce}end;\end{reduce}